\documentclass[11pt,letterpaper]{amsart}

\usepackage[usenames,dvipsnames]{color}
\usepackage{amsthm,amsfonts,amssymb,amsmath,amsxtra, appendix}
\usepackage[all]{xy}
\SelectTips{cm}{}
\usepackage{xr-hyper}
\usepackage{verbatim}
\usepackage[margin=1.25in]{geometry}
\usepackage{mathrsfs}
\usepackage{diagbox}
\usepackage{enumerate}
\usepackage[shortlabels]{enumitem}
\usepackage{tikz}
\usepackage{tkz-euclide}
\usepackage{setspace} 
\usepackage{hyperref}
\usepackage{csquotes}
\usepackage{wrapfig}
\newcommand{\remind}[1]{{\bf ** #1 **}}

\RequirePackage{xspace}
\RequirePackage{etoolbox}
\RequirePackage{varwidth}
\RequirePackage{enumitem}
\RequirePackage{tensor}
\RequirePackage{mathtools}
\RequirePackage{longtable}
\RequirePackage{multirow}

\usepackage{tikz}
\usepackage{tikz-cd}
\usepackage{csquotes}

\usepackage[url=false,doi=false,isbn=false,sorting=nyt,giveninits=true,maxbibnames=4]{biblatex}
\def\ge{\geqslant}
\def\le{\leqslant}
\def\a{\alpha}
\def\b{\beta}
\def\g{\gamma}
\def\G{\Gamma}

\def\D{\Delta}
\def\L{\Lambda}

\def\s{\sigma}
\def\t{\tau}
\def\th{\theta}

\def\l{\lambda}

\def\i{^{-1}}

\def\pr{\mathrm{pr}}
\def\aff{\mathrm{aff}}

\def\QBG{\mathrm{QBG}}

\def\<{\langle}
\def\>{\rangle}

\newcommand{\bG}{\mathbf G}

\newcommand{{\BG}}{\ensuremath{\mathbb {G}}\xspace}

\newcommand{{\BK}}{\ensuremath{\mathbb {K}}\xspace}

\newcommand{\BS}{\ensuremath{\mathbb {S}}\xspace}

\newcommand{\BZ}{\ensuremath{\mathbb {Z}}\xspace}

\newcommand{\CC}{\ensuremath{\mathcal {C}}\xspace}

\newcommand{\CR}{\ensuremath{\mathcal {R}}\xspace}

\newcommand{\Ad}{{\mathrm{Ad}}}

\DeclareMathOperator{\Adm}{Adm}

\DeclareMathOperator{\Gal}{Gal}
\newcommand{\GL}{\mathrm{GL}}

\newcommand{\id}{\ensuremath{\mathrm{id}}\xspace}

\DeclareMathOperator{\Cox}{Cox}

\newcommand{\wt}{\mathrm{wt}}

\def\tW{\tilde W}
\def\ta{\tilde \a}
\def\tb{\tilde \b}

\DeclareMathOperator{\supp}{supp}

\newtheorem{theorem}{Theorem}
\newtheorem{alphatheorem}{Theorem}
\newtheorem{proposition}[theorem]{Proposition}
\newtheorem{lemma}[theorem]{Lemma}
\newtheorem {conjecture}[theorem]{Conjecture}
\newtheorem{corollary}[theorem]{Corollary}

\theoremstyle{definition}
\newtheorem{definition}[theorem]{Definition}
\newtheorem{example}[theorem]{Example}
\newtheorem*{example*}{Example}

\newtheorem{remark}[theorem]{Remark}

\newtheorem*{function*}{Function}

\numberwithin{equation}{section}
\numberwithin{theorem}{section}

\setitemize[0]{leftmargin=*,itemsep=\the\smallskipamount}
\setenumerate[0]{leftmargin=*,itemsep=\the\smallskipamount}

\renewcommand{\to}{%
   \ifbool{@display}{\longrightarrow}{\rightarrow}%
   }
\let\shortmapsto\mapsto
\renewcommand{\mapsto}{%
   \ifbool{@display}{\longmapsto}{\shortmapsto}%
   }
\newlength{\olen}
\newlength{\ulen}
\newlength{\xlen}
\newcommand{\xra}[2][]{%
   \ifbool{@display}%
      {\settowidth{\olen}{$\overset{#2}{\longrightarrow}$}%
       \settowidth{\ulen}{$\underset{#1}{\longrightarrow}$}%
       \settowidth{\xlen}{$\xrightarrow[#1]{#2}$}%
       \ifdimgreater{\olen}{\xlen}%
          {\underset{#1}{\overset{#2}{\longrightarrow}}}%
          {\ifdimgreater{\ulen}{\xlen}%
             {\underset{#1}{\overset{#2}{\longrightarrow}}}
             {\xrightarrow[#1]{#2}}}}%
      {\xrightarrow[#1]{#2}}
   }
\makeatother
\newcommand{\xyra}[2][]{%
   \settowidth{\xlen}{$\xrightarrow[#1]{#2}$}%
   \ifbool{@display}%
      {\settowidth{\olen}{$\overset{#2}{\longrightarrow}$}%
       \settowidth{\ulen}{$\underset{#1}{\longrightarrow}$}%
       \ifdimgreater{\olen}{\xlen}%
          {\mathrel{\xymatrix@M=.12ex@C=3.2ex{\ar[r]^-{#2}_-{#1} &}}}%
          {\ifdimgreater{\ulen}{\xlen}%
             {\mathrel{\xymatrix@M=.12ex@C=3.2ex{\ar[r]^-{#2}_-{#1} &}}}
             {\mathrel{\xymatrix@M=.12ex@C=\the\xlen{\ar[r]^-{#2}_-{#1} &}}}}}%
      {\mathrel{\xymatrix@M=.12ex@C=\the\xlen{\ar[r]^-{#2}_-{#1} &}}}%
   }
\makeatletter
\newcommand{\xla}[2][]{%
   \ifbool{@display}%
      {\settowidth{\olen}{$\overset{#2}{\longleftarrow}$}%
       \settowidth{\ulen}{$\underset{#1}{\longleftarrow}$}%
       \settowidth{\xlen}{$\xleftarrow[#1]{#2}$}%
       \ifdimgreater{\olen}{\xlen}%
          {\underset{#1}{\overset{#2}{\longleftarrow}}}%
          {\ifdimgreater{\ulen}{\xlen}%
             {\underset{#1}{\overset{#2}{\longleftarrow}}}
             {\xleftarrow[#1]{#2}}}}%
      {\xleftarrow[#1]{#2}}
   }
\newcommand{\isoarrow}{%
   \ifbool{@display}{\overset{\sim}{\longrightarrow}}{\xrightarrow\sim}%
   }

\def\le{\leq}
\def\ge{\geq}
\begin{document}

\title[]{Cohen-Macaulayness of Local Models via Shellability of the Admissible Set}

\author[Xuhua He]{Xuhua He}
\address{Department of Mathematics and New Cornerstone Science Laboratory, The University of Hong Kong, Pokfulam, Hong Kong, Hong Kong SAR, China}
\email{xuhuahe@hku.hk}

\author[Felix Schremmer]{Felix Schremmer}
\address{Department of Mathematics and New Cornerstone Science Laboratory, The University of Hong Kong, Pokfulam, Hong Kong, Hong Kong SAR, China}
\email{schremmer@hku.hk}

\author[Qingchao Yu]{Qingchao Yu}
\address{Institute for Advanced Study, Shenzhen University, Nanshan District, Shenzhen, Guangdong, China}
\email{qingchao\_yu@outlook.com}

\thanks{}



\begin{abstract}
We prove that for any dominant cocharacter $\mu$ and any parahoric level $K$, the augmented
admissible set $\widehat{\Adm(\mu)^K}$
in the Iwahori–Weyl group is dual EL-shellable. This resolves a conjecture of G\"ortz \cite{Go01} and provides a new proof 
of the Cohen–Macaulay property for the special fibres of local models with parahoric level structure. In particular, the result settles the previously open cases of residue characteristic $2$ and non-reduced root systems. 

This approach is  characteristic-free and intrinsic to the structure of admissible sets. Moreover, our construction yields an explicit shelling, which translates into an inductive, component-by-component building procedure for the special fibre that preserves Cohen–Macaulayness at each step. 

As a consequence, we obtain the Cohen–Macaulayness of many local models of Shimura varieties considered in the literature, most notably those satisfying the He–Pappas–Rapoport \cite{HPR20} description, as well as the local models characterized by Scholze–Weinstein \cite{SW20} and constructed by Anschütz–Gleason–Lourenço–Richarz \cite{AGLR}. Via the usual local model diagram, these results imply the Cohen–Macaulay property for the corresponding integral models of Shimura varieties whenever available. This gives a new proof that the integral models constructed by Kisin–Pappas–Zhou \cite{KPZ24} are Cohen-Macaulay.
\end{abstract}

\maketitle


\section*{Introduction}

\subsection{Local models and their singularities} 
Shimura varieties occupy a central position in the Langlands program as geometric incarnations of automorphic representations and their associated Galois representations. For most arithmetic applications one is forced to work with integral models of Shimura varieties and to understand their geometry at primes of bad reduction. The local structure of these integral models, most importantly the nature of their singularities, is controlled by certain linear-algebraic moduli problems, known as local models. In a seminal solution to the Langlands program of Harris-Taylor \cite{HT01}, the main objects of interest are certain \enquote{simple} Shimura varieties. One critical feature of these particular Shimura varieties is that their integral models are smooth. This puts a spotlight on the question of what the worst possible singularities for general integral models, or equivalently local models, could be.

Local models are projective schemes over a DVR, encoding the \'etale-local structure of Shimura varieties \cite{KP18} and the moduli stacks of shtukas \cite{AH19} with parahoric level structures. Early work by Deligne–Pappas \cite{DP94}, Chai–Norman \cite{CN90, CN92}, and de Jong \cite{dJ93} linked Shimura variety singularities to these models, later formalized and expanded by Rapoport–Zink \cite{RZ96}. More general (though occasionally less explicit) constructions are due to Zhu \cite{Zhu14} and Pappas–Zhu \cite{PZ13} for tamely ramified groups, with extensions by Levin \cite{Le16}, Lourenço \cite{Lo19+}, Fakhruddin-Haines-Lourenço-Richarz \cite{FHLR} and Richarz \cite{Ri16}.

It was pointed out by Pappas-Rapoport-Smithling \cite[\S 2.1]{PRS} that \enquote{the question of Cohen-Macaulayness and normality of local models is a major open problem in the field}. De Jong \cite{dJ93} pioneered the study of local model singularities, later expanded by Faltings \cite{Fa01}, Görtz \cite{Go01}, Pappas-Rapoport-Smithling \cite{PRS} and Kisin \cite{Ki09}. For large residue characteristics, Pappas-Zhu \cite{PZ13} proved flatness with reduced special fibres and showed that each irreducible component of the special fibre is normal and Cohen–Macaulay. In particular, they deduced normality of local models. For small residue characteristics, it was discovered by Haines-Lourenço-Richarz \cite{HLR} that affine Schubert varieties may not be normal and thus the seminormalization is used in subsequent constructions.

Establishing the Cohen-Macaulayness is more challenging. When the special fibre is irreducible, flatness ensures the model is Cohen–Macaulay. This applies to special parahoric levels and a few other cases classified in \cite{HY24}. In general, the special fibre is the union of multiple Schubert varieties in the partial affine flag variety, indexed by the admissible set in the affine Weyl group. G\"ortz \cite{Go01} proposes to attack the question of Cohen-Macaulayness through a combinatorial problem: more precisely, the dual EL-shellability of these admissible sets. He verified this condition for $\mathrm{GL}_n$ with $n \le 6$ using computer-assisted calculations. 

{\bf Progress since has relied on geometric methods}. In \cite{He13}, the first author related the local model associated with unramified groups and minuscule coweights to the De Concini–Procesi compactification of reductive groups, and deduced Cohen–Macaulayness in these cases from the Cohen–Macaulayness of the group compactification.

A major breakthrough was due to Haines-Richarz \cite{HR23}. For parahoric local models in the sense of Pappas–Zhu \cite{PZ13} (and Levin's extension \cite{Le16}), they proved under minimal assumptions that the entire local model is normal with reduced special fibre and, when $p>2$, also Cohen–Macaulay. This established Cohen–Macaulayness in essentially all tamely ramified cases covered by the Kisin–Pappas–(Zhou) integral models \cite{KPZ24}. The equal-characteristic case was handled via the powerful method of Frobenius splittings on the whole local model, and the mixed characteristic case was deduced from the equal characteristic case by the Coherence conjecture proved by X. Zhu \cite{Zhu14}.

Moving beyond tame ramification, Fakhruddin-Haines-Lourenço-Richarz \cite{FHLR} constructed local models for many wildly ramified groups and proved reduced special fibres and Cohen–Macaulayness in broad generality—highlighting a precise exception for odd unitary groups in residue characteristic $2$. Methodologically, they developed parahoric group schemes over two-dimensional bases for wildly ramified groups and used perfect-geometry tools to analyze singularities of the attached Schubert varieties in positive characteristic.

Ansch\"utz-Gleason-Lourenço-Richarz \cite{AGLR} provide a unique projective flat weakly normal scheme, establishing a constructive solution to an influential conjecture of Scholze-Weinstein \cite{SW20}. Excluding certain root systems in small characteristic, they were able to prove Cohen-Macaulayness of the special fibre using a comparison to the seminormalization of Schubert varieties in equal characteristic. Works by Gleason-Lourenço \cite{GL24} and Cass-Lourenço \cite{CL25} further refined their analysis of the special fibre and establish the Cohen-Macaulayness in more cases (as well as reducedness of the special fibre and its normality in all cases).

Despite these substantial advances, Cohen–Macaulayness remained unsolved for residue characteristic $2$ and non-reduced root systems in the previous literature. Moreover, the geometric methods rely on intricate case analysis to handle small characteristics. This state of affairs left Görtz's original combinatorial conjecture \cite{Go01} as an appealing alternative: if one could prove that the admissible set is dual EL-shellable, then Cohen–Macaulayness would follow from a combinatorial argument. We remark that Gleason-Lourenço \cite[Section~3]{GL24} prove Serre's property $(S_2)$ for the special fibre by linking it to a combinatorial property of admissible sets. It is well known that Cohen-Macaulayness would imply $(S_2)$, and similarly the combinatorial property studied in their paper would be a simple consequence of dual EL-shellability.

\subsection{Dual EL-shellability of admissible sets} This paper resolves G\"ortz’s conjecture, establishing dual EL-shellability of admissible sets, and deduces the Cohen–Macaulayness for the special fibres of the local models with parahoric level structure for all reductive groups. This simplified approach circumvents geometric obstacles and handles all residue characteristics and all parahoric level structures uniformly—including the previously problematic cases of $p=2$ and non-reduced root systems. By the flatness, this also implies Cohen-Macaulayness of the entire local model.

The special fibre of a local model with parahoric level $K$ is a union of affine Schubert varieties in a partial affine flag variety, indexed by the $\mu$-admissible set $\Adm(\mu)^K$. This set was introduced by Kottwitz-Rapoport \cite{KR} in order to describe certain stratifications of Shimura varieties, and further studied by Haines-Ngô \cite{HN02}, and others. 

The closure relations of the affine Schubert varieties in the special fibre of the local model are determined by a partial order on $\Adm(\mu)^K$, known as the Bruhat order. This order can be described group-theoretically using the structure of the Iwahori-Weyl group $\tW$ as a quasi Coxeter group. Since affine Schubert varieties are Cohen-Macaulay (after taking seminormalizations when needed), understanding the structure of the poset $\Adm(\mu)^K$ is essential in establishing Cohen-Macaulayness of the special fibre.

Shellability, a foundational concept in combinatorial geometry and poset topology, originated in the study of polyhedral complexes. A simplicial complex is said to be shellable if its maximal faces can be ordered in such a way that each face intersects the union of the previous faces in a pure and codimension-one manner. This property has significant implications for the algebraic properties of the complex, particularly the regularity of the complex and the Cohen-Macaulayness. Applications extend to the theory of total positivity, where shellability underpins the regularity of the totally positive flag varieties, as demonstrated in \cite{GKL,BH}.

EL-shellability generalizes the notion of shellability for posets via labellings on the set of edges (i.e.\ covers): A poset is EL-shellable if its edges can be labelled such that each interval has a unique increasing and lexicographically minimal chain. This allows for a more flexible and powerful framework for studying the combinatorial and topological properties of posets. One significant application of (dual) EL shellability is in the study of the union of Schubert varieties, as in \cite{Go01}. The shellability requires that the poset has a unique minimal and unique maximal element. Thus we work on the augmented admissible set $\widehat{\mathrm{Adm}(\mu)^K}$, 
consisting of the subposet $\Adm(\mu)^K$ (with Bruhat order) together with an added formal top element $\hat 1$.

\subsection{Main results}
In this paper we prove Görtz’s conjecture. Our main result settles it constructively and in full generality.

\begin{alphatheorem}[Theorem \ref{thm:main}]\label{thmA} For any dominant cocharacter $\mu$ and any parahoric level $K$, the augmented admissible set $\widehat{\mathrm{Adm}(\mu)^K}$ is dual EL-shellable. 
\end{alphatheorem}


Via the general theory of shellable posets \cite{Bj, Go01}, Theorem \ref{thmA} translates directly into a uniform geometric result for local models
--—thereby subsuming and extending the known cases. This includes the local models whose existence was conjectured by Scholze-Weinstein \cite[Conjecture 21.4.1.]{SW20} (who gave a unique characterization) and constructed by Anschütz-Gleason-Lourenço-Richarz \cite{AGLR}. Via the usual local model diagram, these results imply the Cohen-Macaulay property for the corresponding integral models of Shimura varieties whenever available. The most important class of such integral models are the ones constructed by Kisin-Pappas-Zhou \cite{KPZ24}.

\begin{alphatheorem}[Theorem \ref{thm:CM}, Proposition \ref{prop:SWisCM}, Corollary \ref{cor:KPZisCM}]\label{thmB}
    \begin{enumerate}
        \item Any local model satisfying the characterization of He-Pappas-Rapoport \cite[Conjecture~2.1.3]{HPR20} is Cohen–Macaulay.
        \item The local models characterized by Scholze-Weinstein \cite{SW20} and constructed by Anschütz-Gleason-Lourenço-Richarz \cite{AGLR} are Cohen-Macaulay.
        \item The integral models constructed by Kisin-Pappas-Zhou \cite{KPZ24} are Cohen-Macaulay.
    \end{enumerate}
\end{alphatheorem}

The geometric ingredient we need in our proof is the Cohen-Macaulayness of single Schubert varieties. In type A, this result may be proved using a shellability argument \cite{KM05}, but it is not known how to generalize this approach to other types. Cohen-Macaulayness of Schubert varieties may be proved in general using Frobenius splittings.

A notable feature of our approach is that the explicit shelling yields an \emph{explicit recursive construction} of the special fibre: one may add irreducible components one by one in the order dictated by the shelling, preserving the Cohen–Macaulay property at each step. This provides not only a proof of Cohen-Macaulayness, but also a transparent combinatorial blueprint for the geometry of the special fibre. This technique is new even for split group. The possible order for this iterative procedure is detailed in Theorem~\ref{thm:CMConstructive}.

We emphasize that dual EL-shellability implies Cohen–Macaulayness, but the converse implication is not always true (see \cite[p. 183]{Bj}).


\subsection{Methods and Ingredients}

Our proof of Theorem~\ref{thmA} utilizes recent developments regarding structure of admissible sets and affine Weyl groups, as well as classical methods from the representation theory of Hecke algebras.
\begin{itemize}
\item \textbf{Construction of the dual EL-labeling (Section~\ref{sec:2}).} Here, we have to construct a suitable order, combining the following two ingredients:
\begin{enumerate}
\item \textbf{Reflection orders} were introduced by Dyer \cite{Dyer93} based on the theory of Kazhdan-Lusztig polynomials, which play a ubiquitous role in representation theory. A reflection order is a total order on the set of roots of a Coxeter group which is very well behaved for studying chains of Bruhat covers. For our paper, we introduce the notions of Separatedness and $K$-compatibility of reflection orders for the Iwahori-Weyl group, which are well-suited for studying admissible sets for any parahoric level (Lemma~\ref{lem:ref_ord}).
\item \textbf{Component order.} A priori, there is no natural order among the maximal elements of $\Adm(\mu)^K$, which are translation elements of the form $t^{a(\mu)}$. We introduce an order among these translation elements, i.e., the irreducible components of the special fibre of the local model, using the Bruhat order on the elements $a$ of the finite Weyl group $W_0$. This is the order in which we can assemble the special fibre in a component-wise manner, preserving Cohen-Macaulayness at every step. A similar ordering is also used for the irreducible components in level-changing maps in a forthcoming work of Görtz-He-Rapoport \cite{GHR2} on integral Hecke correspondences.
\end{enumerate}
\item
\textbf{Proof of the dual EL-shellability (Section~\ref{sec:4}).}  In order to establish that the construction satisfies the definition of a dual EL-labeling, we use the following two ingredients.
\begin{enumerate}
\item \textbf{The quantum Bruhat graph} was introduced by Brenti-Fomin-Postnikov \cite{BFP} as a particular solution of the quantum Yang-Baxter equations, and is motivated by the quantum Monk formula for the quantum cohomology of flag varieties. The link to the Bruhat order of affine Weyl groups was first described by Lam-Shimozono \cite{LS10}, and has since become a crucial tool in the study of admissible sets. Our paper combines recent developments on the quantum Bruhat graph with a novel construction of particular paths in it (Proposition~\ref{prop:downup}). Our new result generalizes previous work done by Mili\'cevi\'c-Viehmann \cite{MV20} in the context of the so-called \emph{cordial} affine Deligne-Lusztig varieties.
\item \textbf{Acute presentations} are very useful and versatile descriptions of elements of the Iwahori-Weyl group, which we introduce in this work. These presentations are particularly well adapted to the geometry of alcoves in the Bruhat-Tits building, the length function on $\tW$ and its Bruhat order. This notion unifies previous approaches, such as the notion of acute cones of Haines-Ngô and the notion of length positivity by the second author.
\end{enumerate}
\end{itemize}

Since the entire argument can be formulated in the general language of root data and affine Weyl groups, our proof does not require special care of particular root systems or ground fields at any point. We therefore provide a uniform proof of the Cohen-Macaulay property across all cases.

This combination of tools yields a proof that differs fundamentally from earlier geometric approaches. Our approach is  characteristic-free and intrinsic to the structure of admissible sets. The result is a uniform proof that works for all $p$, all parahoric levels, and all root systems—including the previously problematic cases of $p=2$ and non-reduced types.

\subsection{Further Directions}

We expect that the results and approach in this paper find applications beyond local models. 

The shellability property of Weyl groups plays a crucial role in the topological properties of totally nonnegative Schubert varieties (see \cite{GKL, BH}). We expect that the shellabilility of the admissible sets is likely to have implications in the study of total positivity on global Schubert varieties \cite{Zhu, HYZ}. 

We also propose the following conjectures related to the basic loci of Coxeter type \cite{GH15, GHN3}. 

Inside $\Adm(\mu)$, there is an interesting subset $\Cox(\mu)^K$ consisting of certain Coxeter elements. The poset $\Cox(\mu)^K$ is the index set of the EKOR-strata for basic loci of Shimura varieties associated with Shimura data of Coxeter type. We propose the following conjecture.
\begin{conjecture}
The poset $\widehat{\Cox(\mu)^K}=\Cox(\mu)^K \sqcup \{\hat 1\}$ is dual EL-shellable.
\end{conjecture}

The confirmation of this conjecture, together with the Cohen-Macaulayness of single EKOR strata, would imply the Cohen–Macaulayness of the special fibre of basic loci of Coxeter type, extending the reach of our methods to finer stratifications.

{\bf Acknowledgements: } We thank Thomas Haines for valuable suggestions and comments, and Dinakar Muthiah for helpful discussions on reflection orders. XH and FS are partially supported by the New Cornerstone Science Foundation through the New Cornerstone Investigator Program and the Xplorer Prize, and by Hong Kong RGC grant 14300122, both awarded to XH. QY is partially supported by the National Natural Science Foundation of China (grant no. 12501018).

\section{Preliminary}\label{sec:prelim}

\subsection{Shellability}\label{sec:1.1}
Let $(P,\le)$ be a finite poset (partially ordered set). For $w, w' \in P$, write $w'\gtrdot w$ if $w'$ covers $w$ (i.e., $w' > w$ and no $z$ satisfies $w' > z > w$). Denote by $\mathscr{E}(P)$ the set of covering relations (edges).

Let $w,w' \in P$. Define $[w,w'] = \{z\in P\mid w\le z\le w'\}$ and $[w,w')=\{z\in P\mid w\le z < w'\}$. A \textit{downward chain} from $w'$ to $w$ in $P$ of \textit{length} $r$ is a sequence of elements $w'>  w_{r-1} > \ldots > w_1 > w$ in $P$. In this paper, all chains are downward chains. We say a chain is \textit{maximal}, if it is not a subchain of any other chain. 

We say that a finite poset $P$ is \textit{pure}, if all maximal chains in $P$ have the same length. We call this common length the \emph{length} of $P$. If moreover, $P$ contains a unique minimal element $\hat{0}$ and a unique maximal element $\hat{1}$, we say that $P$ is a \textit{graded} poset.

\begin{definition}
    Let $P$ be a graded poset. An {\it edge labeling} of $P$ is a map $\eta : \mathscr{E}(P) \to (\Lambda, \preceq)$, where $(\Lambda, \preceq)$ is a totally ordered set.
\end{definition} 

Given a maximal chain $c : \hat{1} = w_r > w_{r-1} > \dots > w_0 = \hat{0}$, we obtain a word $\eta(c) = (\eta(w_r > w_{r-1}), \dots, \eta(w_1 > w_0))$ in $\Lambda^r$. We say $c$ is {\it label-increasing} if this word is weakly increasing. Two chains are compared lexicographically using the total order on $\Lambda$.

The following definition is due to Bj\"orner–Wachs.
\begin{definition}[{\cite[Definition 4.23]{BW83}}]\label{def:EL}
    An edge labeling $\eta: \mathscr{E}(P) \to (\L,\preceq)$ is said to be a \emph{dual EL-labeling}, if for any interval $[w, w']$ of $P$,
    \begin{enumerate}
        \item There exists a unique label-increasing maximal chain in $[w,w']$;
        \item This chain is lexicographically smallest among all maximal chains in $[w,w']$. 
    \end{enumerate}
If $P$ admits a dual EL-labeling, we say that $P$ is \textit{dual EL-shellable}.
\end{definition}

If $P$ is dual EL-shellable, then its dual poset $P^\vee$ (reversing the order) is said to be EL-shellable. 


\subsection{$N$-Cohen-Macaulay posets}\label{sec:1.2}
G\"ortz introduced the following recursive notion to bridge combinatorial and algebro-geometric properties related to posets.

\begin{definition}[{\cite[Definition 4.23]{Go01}}]\label{def:CM}
Let $Q$ be a pure poset with a unique minimal element $\hat{0}$. For any $z\in Q$, the length of $z$ is the length of any maximal chain in $[\hat{0},z]$. Define $N$-Cohen–Macaulayness recursively:
\begin{enumerate}
    \item If $Q$ has a unique maximal element $x$, then $Q$ is $\ell(x)$-Cohen-Macaulay. 
    \item If $Q$ has maximal elements $x_1, \dots, x_k$ ($k \geq 2$), then $Q$ is $N$-Cohen-Macaulay if all $x_i$ have the same length $N$, and (after reordering) for each $j = 2, \dots, k$, the intersection
  \[
  [\hat{0}, x_j] \cap \bigcup_{i<j} [\hat{0}, x_i]
  \]
  is $(N-1)$-Cohen-Macaulay.
\end{enumerate}
\end{definition}


We use the (combinatorial) $N$-Cohen-Macaulayness of admissible sets to establish the (geometric) Cohen-Macaulayness of the special fibres of local models (see Theorem \ref{thm:CM}). However, since its definition is recursive, it is difficult to verify. It turns out that the notion of dual EL-shellability is more accessible. We now relate these two notions.



\begin{proposition}\label{prop:dualELImpliesCM}
Let $P$ be a graded poset $P$ of length $N+1$. If $P$ is dual EL-shellable, then $P-\{\hat{1}\}$ is $N$-Cohen-Macaulay.
\end{proposition}
\begin{proof}
We first introduce a few more notions. A \emph{coatom} of a graded poset $R$ is an element covered by the maximal element. We say that $R$ \emph{admits a recursive coatom ordering} if the length of $R$ is $1$ or if the length of $R$ is greater than $1$ and there is an ordering  $x_1, x_2, \ldots, x_t $ of the coatoms of $R$ which satisfies:

\begin{enumerate}[(i)]
\item For any $ j = 1, 2, \ldots, t $, $ [\hat{0}_R,x_j] $ admits a recursive coatom ordering in which the coatoms of $[\hat{0}_R,x_j]$ that come first in the ordering are those that are covered by some $ x_i $ where $i < j$. Here, $\hat{0}_R$ is the minimal element of $R$.

\item If $w <x_i,x_j$ with $i<j$, then there exists $ k < j$ and an element $w' $ such that $w\le w'$ and $w' $ is covered by $x_k$ and $x_j$.

\end{enumerate}

We prove the Proposition by induction on $N$. If $N=1$, the statement is trivial. Suppose $N>1$. By \cite[Theorem 3.2]{BW83}, since $P$ is dual EL-shellable, $P$ admits a recursive coatom ordering. Let $x_1,x_2,\ldots,x_t$ be the ordering of the coatoms of $P$ as in the above definition. Let $j>1$. By conditions (i) (ii), the set $[\hat{0},x_j]$ admits a recursive coatom ordering $y_1,y_2,\ldots,y_m,y_{m+1},\ldots, y_{\ell}$ such that the maximal elements of $ [\hat{0},x_j] \cap \left(\cup_{i<j}[\hat{0},x_i]\right)$ are $y_1,y_2,\ldots,y_m$. The recursive coatom ordering of $[\hat{0},x_j]$ restricts to a recursive coatom ordering of $\{x_j\} \sqcup  \left( [\hat{0},x_j] \cap \left(\cup_{i<j}[\hat{0},x_i]\right) \right)$. By induction hypothesis, the set $  [\hat{0},x_j] \cap \left(\cup_{i<j}[\hat{0},x_i]\right)$ is $(N-1)$-Cohen-Macaulay. Hence, $P-\{\hat{1}\}$ is $N$-Cohen-Macaulay by definition \ref{def:CM}.
\end{proof}


\subsection{The Affine Weyl Group}\label{sec:1.3}
Let $\CR = (\Phi, X^*,\Phi^{\vee},X_* ,\D_0)$ be a based root datum, so that $X^\ast$ and $X_\ast$ come with a natural pairing $\langle \cdot,\cdot\rangle$. Let $W_0$ be its Weyl group and $\BS_0 = \{s_\alpha\mid \alpha\in \Delta_0\}$ be the set of simple reflections of $W_0$. Let 
$$\tW = X_*\rtimes W_0 = \{ t^{\l}z \mid \l \in X_*, z\in W_0\}$$ be the extended affine Weyl group. 

Let $\Phi = \Phi^+\sqcup \Phi^-$ be the sets of positive and negative roots respectively. The set of affine roots is defined as $\Phi_{\aff} = \Phi \times \BZ$. The set of positive affine roots is defined as $\Phi_{\aff}^+ = (\Phi^+ \times \BZ_{\ge0} )\sqcup (\Phi^- \times \BZ_{\ge1})$, and its complement are the negative affine roots. The \emph{simple affine roots} are
\begin{align*}
    \Delta_{\aff} =\{(\alpha,0)\mid \alpha \in \D_0\} \bigcup \{(-\theta,1)\mid \theta\text{ highest root of an irred.\ component of }\CR\}.
\end{align*}
By convention, we view $\Phi$ as a subset of $\Phi_{\aff}$ via the embedding $\a \to (\a,0)$. In particular, $\Delta_0\subseteq \Delta_{\aff}$.

For any $w = t^{\l}z \in   \tW$, the action on $\Phi_{\aff}$ is given by $$\tilde{\a} = (\a,k)\in \Phi\times\BZ\mapsto w(\a,k) = (z(\a), k-\<\l,z(\a)\>).$$ The affine reflection corresponding to $\tilde\a \in \Phi_{\aff}^+$ is $s_{\tilde{\a}} = s_{\a} t^{k\a^{\vee}} \in \tW$. Then the set of simple affine reflections is $\tilde{\BS} = \{s_{\tilde\alpha}\mid \tilde \alpha\in \Delta_{\aff}\}$. Then the pair $(\tW, \tilde{\BS})$ is a \emph{quasi-Coxeter group}, It behaves like a Coxeter group except for the possible presence of finitely many length-zero elements. We view $W_0$ as a natural subset of $\tW$, and get $\BS_0\subseteq \tilde{\BS}$ in this way.


\subsection{The admissible set}\label{sec:1.4}
Let $\mu \in X_*^+$ be a (not necessarily minuscule) dominant cocharacter. The $\mu$-admissible set is defined by
$$\Adm(\mu) = \{ w\in \tW \mid w \le t^{ z (\mu)} \text{ for some } z\in W_0\}.$$ 

For any subset $K$ of $\D_{\aff}$, let $W_K$ be the subgroup of $\tW$ generated by $\{s_{\tilde\alpha}\mid \ta \in K\}$ and $\tW^K$ be the set of minimal length coset representatives of $\tW/W_K$.  We say that $K$ is {\it spherical} if $W_K$ is finite. For spherical subset $K$ of $\D_{\aff}$, we define $$\Adm(\mu)^K = \Adm(\mu) \cap \tW^K.$$ By  \cite[Theorem 6.1]{He16} and \cite[Proposition 5.1]{HH17}, $\Adm(\mu)^K=W_K \Adm(\mu) W_K \cap \tW^K$. Geometrically, $\Adm(\mu)^K$ indexes the Schubert cells occurring in the special fibre of the local model with the corresponding parahoric level structure.

Note that $\Adm(\mu)^K$ has a unique minimal element, which we denote by $\tau$. By \cite[Proposition 2.1]{HN17}, the maximal elements in $\Adm(\mu)^K$ are $t^{ \mu'}$ with $\mu' \in W_0(\mu)$ such that $t^{\mu'} \in \tW^K$. For any Bruhat interval $[w,w']$ in $\tW$, denote $[w,w']^K = [w,w']\cap \tW^K$. By \cite[Theorem 2.5.5]{BB}, $[w,w']^K$ is a graded poset. Then $\Adm(\mu)^K = \bigcup_{\mu'} [\tau,t^{\mu'}]^K$ is a graded poset. 

Define $\widehat{\Adm(\mu)^K}=\Adm(\mu)^K \sqcup \{\hat 1\}$, where $\hat{1}$ is the added unique maximal element. It is clear that  $\widehat{\Adm(\mu)^K}$ is a graded poset.

\section{Shellings of the admissible set}\label{sec:2}



We fix a spherical subset $K\subseteq\D_{\aff}$. To prove the dual EL-shellability of $\widehat{\Adm(\mu)^K}$, we construct a suitable edge labeling. We handle the following two types of edges. 

\begin{enumerate}
    \item {\it Interior edges} between elements of $\Adm(\mu)^K$: We use Dyer's theory of reflection orders on affine roots $\Phi_{\aff}^+$.
    \item {\it Augmented edges} $\hat{1} \gtrdot t^{a(\mu)}$: We introduce new labels $\eta_a$ and use the Bruhat order on $a$.
\end{enumerate}

The key part is to choose a suitable total order on $\Phi^+_{\aff}$ that is ``compatible'' with the level $K$ and the Bruhat order of the elements $a \in W_0$. 

\subsection{Reflection order}\label{sec:2.1}
For any covering relation $w' \gtrdot w$ in $\tilde{W}$, there exists a unique positive affine root $\tilde{\alpha}$ such that $w' = w s_{\tilde{\alpha}}$. We label the interior edge $w' \gtrdot w$ by the affine root $\tilde \a$. 


\begin{definition}
    A total order $\preceq$ on $\Phi_{\mathrm{aff}}^+$ is called a {\it reflection order} if for any $\tilde{\alpha}, \tilde{\beta} \in \Phi_{\mathrm{aff}}^+$ and $a, b > 0$ such that $\tilde{\alpha} \prec \tilde{\beta}$ and $a\tilde{\alpha} + b\tilde{\beta} \in \Phi_{\mathrm{aff}}^+$, we have
\[
\tilde{\alpha} \prec a\tilde{\alpha} + b\tilde{\beta} \prec \tilde{\beta}.
\]
\end{definition} 


For our purpose, we also need to take into account the level $K \subseteq \D_{\aff}$. We introduce some notations. Let $\Phi_K$ be the root subsystem of $\Phi_{\aff}$ spanned by affine roots in $K$ and let $\Phi_{K}^+ = \Phi_{K}\cap\Phi_{\aff}^+$. 

The following result was established by Dyer \cite[Proposition 4.3]{Dyer93}, based on the theory of Hecke algebras and $R$-polynomials of Kazhdan-Lusztig \cite{KL79}.

\begin{proposition}\label{prop:Dyer}
Let $K \subseteq \Delta_{\mathrm{aff}}$. Suppose $\preceq$ is a reflection order such that every root in $\Phi_K^+$ precedes every root in $\Phi_{\mathrm{aff}}^+ \setminus \Phi_K^+$. Then the induced labeling on any interval $[w, w']^K \subseteq \tilde{W}^K$ is a dual EL-labeling.
\end{proposition}

Thus, for intervals contained entirely within $\Adm(\mu)^K$, we can use any reflection order with this \enquote{$K$-compatibility} property. To extend this labeling to the entire augmented poset $\widehat{\Adm(\mu)^K}$, we need to choose our reflection order carefully.


\subsection{Constructing a special reflection order}\label{sec:2.2}
We now construct a reflection order tailored to our needs.  We pay special attention to the set $K\cap \Delta_0$.

\begin{definition}
A positive affine root is called {\it of type I} if it is in $\Phi_{K}^+$ or is in $(\Phi^-\setminus \Phi_{K\cap \Delta_0}^-) \times \mathbb Z_{>0}$. It is called {\it of type II} if it is not of type I. 
\end{definition}

Now we establish the following key lemma. It ensures that our reflection order respects both the parahoric structure (via $K$-compatibility) and the dichotomy between type I and type II roots needed for the augmented edges.

\begin{lemma}\label{lem:ref_ord}
    There exists a reflection order $\preceq$ on $\Phi_{\aff}^+$ satisfying the following two properties:
    \begin{enumerate}
    \item Separation property: Every type I root precedes every type II root.
    \item $K$-compatibility property: Every root in $\Phi_K^+$ precedes every root in $\Phi_{\mathrm{aff}}^+ \setminus \Phi_K^+$.
    \end{enumerate}
\end{lemma}

\begin{proof}
We first construct a reflection orders satisfying (1). 

Let $w$ be the longest element of the subgroup generated by $K \cap \Delta_0$. Let $\l$ be a dominant cocharacter $\lambda$ whose stabilizer in $\Delta_0$ is equal to $K \cap \Delta_0$. Then $w (\l) = \l$. For any $n\ge 0$ and $\tilde \a=(\a, k) \in \Phi^+_{\aff}$, we have $$(wt^{n\lambda})^{-1} (\ta)=(w^{-1} (\a), k+\langle n w (\l), \alpha\rangle)=(w^{-1}(\a), k + \langle n\lambda, \alpha\rangle).$$
Thus, $(wt^{n\lambda})^{-1}(\ta) \in \Phi_{\aff}^-$ if and only if $(\a, k) \in \Phi_{K \cap \Delta_0}^+$ or $\a \in (\Phi^-\setminus \Phi_{K\cap \Delta_0}^-)$ and $k+n\langle \lambda,\alpha\rangle\leq 0$. Hence $\bigcup_{n \ge 0} \{\tilde\alpha\in \Phi_{\aff}^+\mid (wt^{n\lambda})^{-1}(\ta) \in \Phi_{\aff}^-\}$ is exactly the set of all type I positive affine roots. By \cite[Proposition~3.4]{CP98}, there exists a reflection order $\preceq$ satisfying (1).


Suppose the reflection order $\preceq$ does not satisfy (2). Set $$D(\preceq) := \{\tilde\alpha\in \Phi_K^+\mid \text{ there exists } \tilde\beta\in \Phi_{\aff}^+\setminus\Phi_{K}^+\text{ with  }\tilde\beta\prec\tilde\alpha\}.$$
From the definition of reflection orders, $D(\preceq)$ must contain at least one simple affine root $\ta \in K$. Denote $\tilde s = s_{\ta}$. Following \cite[Proposition~2.5]{Dyer93}, we define a reflection order $\preceq_{\tilde s}$ as  
\begin{align*}
    \tilde \beta\preceq_{\tilde s}\tilde \gamma\iff \begin{cases}\tilde\beta=\tilde\alpha,&\text{ or}\\
    \tilde\beta \neq \tilde \a, \tilde \b \preceq \tilde\gamma,  \tilde\alpha \prec \tilde \g,&\text{ or}\\
    \tilde\beta \neq \tilde \a, \tilde \gamma\prec\tilde\alpha, \tilde s(\beta)\prec \tilde s(\gamma).\end{cases}
\end{align*}
It is straightforward to check that $\preceq_{\tilde s}$ also satisfies (1) and that $\sharp D(\preceq_{\tilde s})<\sharp D(\preceq)$.




Now we replace $\preceq$ by $\preceq_{\tilde s}$. 
Since $K$ is spherical, this procedure stops after finitely many steps, and we end up with a reflection order satisfying both (1) and (2). \end{proof}

\subsection{Labeling on the augmented edges}\label{sec:2.3}

Let $J = \{ i \in \D_0 \mid \<\mu,\a_i\> = 0 \}$. Let $W^J$ be the set of minimal representatives in $ W_0 / W_J$. Let $W^{J,K} = \{ a \in W^J \mid t^{a(\mu)} \in \tW^K \}$. By \cite[Proposition 2.1]{HN17}, the map $a \mapsto t^{a(\mu)}$ gives a bijection between $W^{J, K}$ and the maximal elements in $\Adm(\mu)^K$.

To each $a \in W^{J,K}$, we associate a symbol $\eta_a$. Set $\Lambda = \Phi_{\mathrm{aff}}^+ \sqcup \{\eta_a \mid a \in W^{J,K}\}$. We choose a total order on $\{\eta_a \}$ refining the (induced) Bruhat order on the elements $a \in W^{J,K}$. We extend the reflection order $\preceq$ from Lemma \ref{lem:ref_ord} to a total order on $\Lambda$ by declaring:
\[
\tilde{\alpha} \prec \eta_a \prec \tilde{\beta} \quad \text{for all } a \in W^{J,K}, \text{ type I roots } \tilde{\alpha}, \text{ and type II roots } \tilde{\beta}.
\]

Define $\hat{\eta}: \mathscr{E}(\widehat{\Adm(\mu)^K}) \to (\L,\preceq)$ by $\hat{\eta}(w s_{\tilde{\alpha}} \gtrdot w) = \tilde{\alpha}$ and $\hat{\eta}(\hat{1} \gtrdot t^{a(\mu)}) = \eta_a$.

The main result of this paper is the following.

\begin{theorem}\label{thm:main}
The edge labeling $\hat{\eta}$ on $\widehat{\Adm(\mu)^K}$ is a dual EL-labeling. More generally, given a subset $C\subseteq W^{J,K}$ satisfying the condition that for all $a\in C$ and $b\in W^{J,K}$ with $\eta_a\succ \eta_b$ we have $b\in C$, then the restriction of $\hat{\eta}$ to
\begin{align*}
    \{w\in \tW^K\mid \exists a\in C:~ w\leq t^{a(\mu)}\}\sqcup \{\hat 1\}
\end{align*}
is a dual EL-labeling.
\end{theorem}

\begin{example}
    Consider the general linear group $\GL_3$. We naturally identify its character and cocharacter lattice $X_\ast, X^\ast$ with $\mathbb Z^3$. The set of (co)roots is given by
    \begin{align*}
        \Phi = \Phi^\vee = \{\pm(1,-1,0), \pm (1,0,-1), \pm (0,1,-1)\}.
    \end{align*}
    The simple roots are given by $\alpha_1 = (1,-1,0)$ and $\alpha_2 = (0,1,-1)$. We write $\tilde{\BS} = \{s_0, s_1, s_2\}$ with $s_1 = s_{\alpha_1}, s_2 = s_{\alpha_2}$ and $s_0=t^{(1,0,-1)}s_1s_2s_1$. Set $K = \{\a_1\}\subseteq \Delta_{\aff}$. Let $\mu = (1,0,0)$. Then $J = \{\alpha_2\}$. The admissible set $\Adm(\mu)^K$ contains a unique element of length zero, which we denote by $\tau = s_1 s_2 t^{(0,0,1)}$.  The poset $\widehat{\Adm(\mu)^K}$ can be drawn as follows:
    
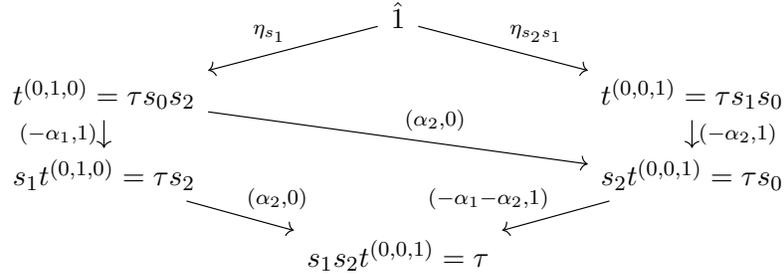
\begin{align*}
        \begin{tikzcd}[ampersand replacement=\&,column sep=3em, row sep=1em]
        \&\hat 1\ar[dl,"{\eta_{s_1}}"']\ar[dr, "{\eta_{s_2 s_1}}"]\\
        t^{(0,1,0)} = \tau s_0 s_2\ar[d,"{(-\alpha_1,1)}"']\ar[drr,"{(\alpha_2,0)}"]\&\&t^{(0,0,1)} = \tau s_1 s_0\ar[d,"{(-\alpha_2,1)}"]\\
        s_1 t^{(0,1,0)} =  \tau s_2\ar[dr,"{(\alpha_2,0)}"]\&\&s_2 t^{(0,0,1)} = \tau s_0\ar[dl,"{(-\alpha_1-\alpha_2,1)}"']\\
        \&s_1 s_2 t^{(0,0,1)} = \tau
        \end{tikzcd}
    \end{align*}
    Our edge labels are totally ordered as follows:
    \begin{align*}
        (-\alpha_2,1) \prec (-\alpha_1-\alpha_2,1) \prec \eta_{s_1}\prec \eta_{s_2 s_1}\prec (-\alpha_1,1) \prec (\alpha_2,0).
    \end{align*}
    We leave it as an exercise for the reader to verify that this is a dual EL-labeling.


\end{example}

\section{Quantum Bruhat Graph and Bruhat Order on $\tW$}\label{sec:3}

\subsection{Quantum Bruhat graphs}\label{sec:3.1} 
Understanding the Bruhat order on the affine Weyl group $\tilde{W}$ is central to our study of admissible sets. A powerful tool for navigating this complexity is the Quantum Bruhat Graph (QBG), introduced by Fomin, Gelfand and Postnikov in \cite{FGP97}. 

Set $\rho = \frac{1}{2}\sum_{\a\in \Phi^+} \a$. The quantum Bruhat graph $\operatorname{QBG}(\mathcal{R})$ is a directed, weighted graph with:
\begin{itemize}
    \item Vertices: Elements of $W_0$.
    \item Edges: Two types, labeled by positive roots $\alpha \in \Phi^+$:
\begin{enumerate}
    \item Bruhat edges: $w \rightharpoonup w s_\alpha$ when $w s_\alpha > w$.
    \item Quantum edges: $w \rightharpoondown w s_\alpha$ when $\ell(w s_\alpha) = \ell(w) + 1 - \langle \alpha^\vee, 2\rho \rangle$.
\end{enumerate}
     \item Weights: Bruhat edge $w \rightharpoonup w s_\alpha$ has weight 0; quantum edge $w \rightharpoondown w s_\alpha$ has weight $\alpha^\vee$.
\end{itemize}

We use $\rightharpoonup$ and $\rightharpoondown$ to denote Bruhat edges and quantum edges, respectively, emphasizing that the edges are going up or down in the Bruhat order. We also use $\to$ for both Bruhat and quantum edges. 

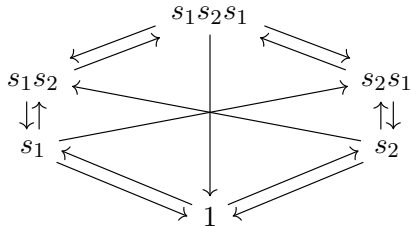
\begin{figure}[ht]\centering
\def\seshift{0.5ex}
\begin{align*}
\begin{tikzcd}[ampersand replacement=\&,column sep=3em,row sep=1em]
\&s_1 s_2 s_1\ar[ddd]
\ar[dl,shift right=\seshift]\ar[dr,shift left=\seshift]\\
s_1 s_2\ar[ur,shift right=\seshift]\ar[d,shift right=\seshift]\&\&
s_2 s_1\ar[ul,shift left=\seshift]\ar[d,shift left=\seshift]\\
s_1\ar[u,shift right=\seshift]\ar[urr]\ar[dr,shift right=\seshift]\&\&s_2\ar[u,shift left=\seshift]\ar[ull]\ar[dl,shift left=\seshift]\\
\&1\ar[ru,shift left=\seshift]\ar[lu,shift right=\seshift]
\end{tikzcd}
\end{align*} 
\end{figure}
\begin{example}For type $A_2$ with simple roots $\alpha_1, \alpha_2$, the QBG has six vertices. The longest element $w_0 = s_1 s_2 s_1$ has quantum edges to $s_1 s_2$ and $s_2 s_1$ with weights $\alpha_1^\vee$ resp.\ $\alpha_2^\vee$. 
\end{example}

It is easy to see that $\QBG(\CR)$ is path-connected. For $u,v\in W_0$, the length of any shortest path from $u$ to $v$ is denoted by $d(u,v)$. The {\it weight of a path} is the sum of weights of its edges. The following properties on the quantum Bruhat graph was established by Postnikov \cite[Lemma 1]{Pos05} and Mili\'cevi\'c-Viehmann \cite[Equation~(4.3)]{MV20}. 
\begin{lemma}\label{lem:Pos}
Let $u,v\in W_0$. Then 
\begin{enumerate}
    \item All shortest paths from $u$ to $v$ in $\mathrm{QBG}(\CR)$ have the same weight. We denote this weight by $\wt(u, v)$.
    \item Any path from $u$ to $v$ has weight $\ge$ $\wt(u,v)$, the equality holds if and only if the path is shortest.
    \item $\wt(u ,v) = 0$ if and only if $u \le v$.
\end{enumerate}
\end{lemma}

By definition of the weight function, we have 
\begin{align*}
\wt(u,u'') \le \wt(u,u')+ \wt(u',u'') \text{ for any } u,u',u''\in W_0.
\end{align*}

Now we prove the existence of shortest paths of specific forms. A \emph{down-up type} path in $\text{QBG}(\CR)$ is a path $u_1 \to u_2 \to \cdots \to u_r$ where there is no $i$ such that $u_{i-1}<u_i$ and $u_i > u_{i+1}$. We define \emph{up-down type} path similarly. In particular, paths consisting of only Bruhat edges, or only quantum edges, are both of down-up type and of up-down type.

\begin{proposition}\label{prop:downup}
Let $u,v \in W_0$ with $u\ne v$. Then there exists a shortest path from $u$ to $v$ of down-up type. Similarly, there exists a shortest path from $u$ to $v$ of up-down type. 
\end{proposition}

\begin{remark}
The special case where $v=1$ was first proved in \cite[Proposition 4.11]{MV20}.
\end{remark}

\begin{proof}
We prove the existence of down-up type paths. The other claim follows from the duality $w \mapsto w_0 w$.

First consider the special case $d(u,v)=2$. Then there exists a shortest path $p$ of length $2$ from $u$ to $v$, and the statement is automatic unless $p$ is of the form
\begin{align*}
    p:u\rightharpoonup vs_\gamma\rightharpoondown v
\end{align*}
for a root $\gamma$ such that $\wt(u,v) = \gamma^\vee$. In this case, $vs_\gamma$ is a length additive product and we get a reduced word for $vs_\gamma$ by concatenating a reduced word for $v$ with a reduced word for $s_\gamma$. Now $u\lessdot vs_\gamma$ implies we get a reduced word for $u$ by deleting one letter of the reduced word for $vs_\gamma$. 

If $u = u_1 s_\gamma$ with $\ell(u) = \ell(u_1) + \ell(s_\gamma)$ and $u_1\lessdot v$, we get an alternative path $u\rightharpoondown u_1\rightharpoonup v$, which is of down-up type as desired. 

Otherwise, $u = v u_2$ with $\ell(u) = \ell(v) + \ell(u_2)$ and $u_2\lessdot s_\gamma$. We have $u_2 = s_\gamma s_\alpha$
for some $\alpha\in \Phi^+$. Since $u_2 \le s_\g$, the root $s_\gamma(\alpha)$ is negative. By \cite[Lemma 3.1 and Remark 3.3]{Sch24}, the fact that $\gamma$ occurs as part of the quantum edge $vs_\gamma\rightharpoondown v$ implies that $\langle \gamma^\vee,\alpha\rangle=1$.

Let $\beta = s_\alpha(\gamma)$. Then $\beta^\vee = s_\alpha(\gamma^\vee) = \gamma^\vee-\alpha^\vee$. From \cite[Lemma~4.3]{BFP}, we get the inequalities
\begin{align*}
    \ell(vs_\alpha)& \leq \ell(v)+\ell(s_\alpha) \leq \ell(v) + \langle \alpha^\vee,2\rho\rangle-1,
    \\\ell(u) \leq& \ell(vs_\alpha) + \ell(s_{\beta}) \leq \ell(vs_\alpha) + \langle \beta^\vee,2\rho\rangle-1.
\end{align*}
Studying the path $p$, we get the observation
\begin{align*}
    \ell(u) = \ell(vs_\gamma)-1 = \ell(v) + \langle \gamma^\vee,2\rho\rangle-2.
\end{align*}
Combining the above three formulas with the identity $\gamma^\vee = \alpha^\vee + \beta^\vee$, we get a desired path
\begin{align*}
    u = v s_\gamma s_\alpha = v s_\alpha s_{\beta}\rightharpoondown vs_\alpha \rightharpoondown v.
\end{align*}

This finishes the proof for the case $d(u,v)=2$.

In general, given any shortest path from $u$ to $v$ that is not of down-up type, then we can find a subpath of length two that consists of a Bruhat edge followed by a quantum edge. We replace this subpath, as above, by a length two path of down-up type without changing the rest of the path.

Through this procedure, the total number of quantum edges in our path may stay constant or increase, but never decrease. Moreover, if this number stays constant, then the positions of quantum edges in the path are moved closer to the start. Therefore, iterating this procedure will eventually result in a path from $u$ to $v$ of down-up type.
\end{proof}


\subsection{Acute presentation}\label{sec:3.2} 


In this subsection, we introduce a particularly convenient class of factorizations of Iwahori-Weyl group elements $w = xt^{\lambda} y\in \tW$ with $x,y\in W_0$ and $\lambda\in X_\ast$, called \emph{acute presentations}. These are particularly well adapted to the geometry of alcoves in the Bruhat-Tits building, the length function on $\tW$ and its Bruhat order. The elements most important to our work, such as translation elements $t^{a(\mu)}$ or elements in $\tW^K$ for $K\subseteq \Delta_{\aff}$, come with particular acute presentations that will turn out to be very useful. This section unifies previous approaches, such as the notion of acute cones of Haines-Ngô \cite{HN02} and the notion of length positivity by the second author. We will explain how results in the previous literature can be expressed very elegantly using the language of acute presentations.

\begin{definition} Let $w \in \tilde{W}$. A factorization $w = x t^\lambda y$ with $x, y \in W_0$, $\lambda \in X_*$ is called an {\it acute presentation} if
\[
\delta_{\Phi^-}(x(\a )) + \langle \lambda, \alpha \rangle - \delta_{\Phi^-}(y^{-1}(\alpha)) \geq 0 \quad \text{for all } \alpha \in \Phi^+,
\]
where $\delta_{\Phi^-}(\beta)$ is the Kronecker symbol, i.e., $\delta_{\Phi^-}(\beta)=1$ if $\beta\in \Phi^-$ and $0$ otherwise.

\end{definition}

The notion of an acute presentation coincides with that of a \emph{length‑positive element} introduced in \cite[Definition~2.2]{Sch24}: Explicitly, $w=xt^{\lambda }y$ is acute precisely when $y^{-1}\in \operatorname{LP}(w)$.
A geometric interpretation of acute presentations can be given via the \emph{acute cones} of Haines-Ngô  \cite[\S5]{HN02}.

The extended affine Weyl group $\tW$ acts on the real vector space $V = X_* \otimes_{\mathbb{Z}} \mathbb{R}$ by affine transformations: $t^\lambda w \cdot v = \lambda + w(v)$.
By definition, alcoves are connected components of $V - \bigcup_{\tilde{\a}}H_{\tilde{\a}}$, where $\tilde{\a}$ runs over the set of affine roots $\Phi_{\aff}$. By convention, the base alcove is defined as 
$$\mathfrak{a} = \{ v\in V\mid 0 < \<v,\a \> < 1\text{ for every }\a\in \Phi^+\}.$$
Let $C^+\subseteq V$ be the dominant Weyl chamber and $X_*^+ \subseteq X_*$ be the set of dominant cocharacters. Let $\ell$ be the length function and $\le$ be the Bruhat order with respect to the base alcove $\mathfrak{a}$. 


\begin{definition}
Let $z \in W_0$. The \emph{acute cone in the $z$-direction} is $$\CC(\mathfrak{a},z)=\{w \in  \tW \mid w(\mathfrak{a}) \subseteq H^{z+} \text{ for all root hyperplanes }H \text{ with } \mathfrak{a} \subseteq H^{z+}\},$$ where $H^{z+} = H + z(C^+)$ is the connected component of $V \setminus H$ that contains any sufficiently deep alcoves in the Weyl chamber $z(C^+)$. 
\end{definition}

There is a natural bijection between acute presentations of $w$ and acute cones containing $w$.

\begin{proposition}\label{prop:acute}
Let $w =xt^{\l}y \in\tW$ with $x,y\in W_0$ and $\l\in X_*$. Then $ xt^{\l}y$ is an acute presentation of $w$ if and only if $w \in \CC(\mathfrak{a},x)$.
\end{proposition}
\begin{proof}
By definition, 
$\mathfrak{a} = \{ v\in V \mid -\delta_{\Phi^-}(\b) < \<v,\b\> <  1-\delta_{\Phi^-}(\b) \text{ for every }\b\in\Phi^+   \}.$
Then by definition, $w\in\CC(\mathfrak{a},x)$ if and only if
$w(\mathfrak{a}) \subseteq \bigcap_{  \b\in x(\Phi^+)  }   \{v\in V\mid \<v,\b\> \ge -\delta_{\Phi^-}(\b)   \}  .$
For any $v \in V$ and $\a \in \Phi^+$, we have $\<w(v) , x (\a)\> =\< x(\l) + xy( v)  ,x(\a)\> = \<\l,\a\>+ \<v,y^{-1}(\a)\>  $. Hence $w\in\CC(\mathfrak{a},x)$ if and only if 
$\<\l,\a\>  - \delta_{\Phi^-}(y^{-1}(\a))   \ge -\delta_{\Phi^-}(x(\a)) $
for any $\a\in \Phi^+$.
\end{proof}


\begin{corollary}[{\cite[Corollary 2.11]{Sch23}}]\label{cor:lengthViaAcutePres}
Let $w = xt^{\l}y$ with $x,y\in W_0$ and $\l\in X_*$. Then $\ell(w)\ge \ell(x ) + \<\l,2\rho\> - \ell(y)$, the equality holds if and only if $ xt^{\l}y$ is an acute presentation of $w$.  
\end{corollary}

Every $w\in \tilde{W}$ admits a unique {\it standard presentation}
$w=x_{0}t^{\l_0}y_{0}$ with $\l_0\in X_{*}^{+}$ and
$t^{\l_0}y_{0}$ of minimal length in its coset $W_{0}\backslash \tilde{W}$
(see e.g.\ \cite[\S9.1]{He14}). In this presentation $w(\mathfrak{a})$ lies in
the Weyl chamber $x_{0}(C^{+})$, hence $w\in \mathcal{C}(\mathfrak{a},x_{0})$ and
$x_{0}t^{\l_0}y_{0}$ is acute. In general, an element may belong to several
acute cones; it belongs to exactly one acute cone if and only if $w(\mathfrak{a})$
lies in the shrunken Weyl chamber, or equivalently, if $w$ lies in the lowest
two‑sided Kazhdan–Lusztig cell of $\tilde{W}$ (see \cite[Proposition~2.15]{Sch23}).
We shall not need this fact in the present paper.

We have the following description of the Bruhat orders in $\tW$ via the acute presentations. 

\begin{theorem}[{\cite[Theorem 4.2; Lemma 4.3]{Sch24}}]\label{thm:Felix}
Let $w\in\tW$ and let $  x t^{\l}y$ be an acute presentation of $w$. Let $w' \in \tW$. Then $w\le w'$ if and only if there exists $x',y' \in W_0$ and $\l'\in X_*$ with $w' = x' t^{\l'} y'$ such that
$$\wt(x,x') + \wt(y'^{-1},y^{-1}) \le \l' - \l.$$
In this case, one may choose $x',y',\lambda'$ such that $w' = x' t^{\lambda'} y'$ is an acute presentation.
\end{theorem}

As an application, we can describe the admissible set via the quantum Bruhat graph.
\begin{proposition}[{\cite[Proposition~4.12]{Sch24}}]\label{prop:admMembershipTest}
    Let $w = xt^{\lambda} y$ be an acute presentation and $\mu\in X_\ast^+$. Then $w\in \Adm(\mu)$ if and only if the inequality
    \begin{align*}
        \wt(x, y^{-1})\leq \mu-\lambda
    \end{align*}
    holds true.
\end{proposition}

We have the following description of acute presentations of translation elements.

\begin{lemma}\label{lem:trans}
Let $z \in W_0$, $\l\in X_*^+$ and $J=\{\a \in\D_0\mid \<\l,\a\> = 0\}$.  Then the acute presentations of $t^{z(\l)}$ are $z u t^{\l} u^{-1} z$, where $u$ runs over $W_J$.
\end{lemma}
\begin{proof}
Let $u \in W_0$. We have $t^{z(\l)} = z u t^{u^{-1}(\l)} u^{-1} z^{-1}$. Then
\begin{align*}
    &z u t^{u^{-1}(\l)} u^{-1} z^{-1} \text{ is an acute presentation of }t^{z(\l)}\\
    \iff&  \delta_{\Phi^-}(zu(\a)) + \<u^{-1}(\l), \a\> - \delta_{\Phi^-}( zu(\a) ) \ge 0 \text{ for any }\a\in \Phi^+\\
    \iff&u^{-1}(\l)\text{ is dominant}\\
    \iff & u\in W_J. \qedhere
\end{align*}
\end{proof}

Finally, elements in $\tW^K$ come with particularly useful acute presentations.

\begin{proposition}\label{prop:acute-pres}
Let $w \in\tW^K$. Then there exists an acute presentation $w = x t^{\l}y$ such that $y(\a )\in \Phi^- $ for all $\ta = (\a,k) \in K$.
\end{proposition}

\begin{remark}\label{rmk:WK}
For any $w\in \tW$ and $\tilde\a\in \Phi^+_{\aff}$, $ws_{\tilde\a}>w$ if and only if $w(\tilde{\a})\in \Phi^+_{\aff}$. As a result, for any $ \l \in X_* $, we have (see for example \cite[\S1.3, (1.1)]{HY24})
\begin{align*}
    t^{\l}\in \tW^{K} \iff   \<\l,\a \> \le 0 \text{ for any } \ta = (\a,k) \in K.
\end{align*}
In particular, in the situation of Proposition~\ref{prop:acute-pres}, $t^{y \i(\l)} \in \tW^K$. 
\end{remark}

\begin{proof}
We write $\pr$ for the natural projection maps $\pr: \tW \longrightarrow W_0$ and $\Phi_{\aff} \to \Phi$. 

Let $w_{0,K}\in W_K$ denote the longest element of the group generated by $K$. Then $w_1 := w w_{0,K}$ is a length additive product.
Let $w_1 = x_1 t^{\l_1} y_1$ be any acute presentation. Then $w_1\in\CC(\mathfrak{a},x_1)$ by Proposition \ref{prop:acute}. By \cite[\S5]{HN02}, any minimal gallery from $\mathfrak{a}$ to $w_1(\mathfrak{a})$ is in the $x_1$-direction. Since $\ell(w_1) = \ell(w) + \ell(w_{0,K})$, there is a minimal gallery from $\mathfrak{a}$ to $w_1(\mathfrak{a})$ which passes through $w(\mathfrak{a})$. By \cite[Lemma 5.3]{HN02}, this forces $w \in \CC(\mathfrak{a},x_1)$. By Proposition \ref{prop:acute-pres}, we get an acute presentation $w = x_1 t^{\l}y$ for some $\l$ and $y$. 

It is clear that $y = y_1\pr(w_{0,K})$. By \cite[Lemma~2.9]{Sch23}, we have $y_1(\a) \in \Phi^+$ for all $\ta = (\a,k) \in K$. Note that $y(\a ) = y_1 \pr(w_{0,K}(\ta ) ) = y_1 \pr(-\tb)  = -  y_1 (\pr(\tb)) $. Here, $\tb= - w_{0,K}(\ta) \in K $. Hence $y(\a) \in \Phi^-$.
\end{proof}

\subsection{Top two layers of $\Adm(\mu)^K$}

Let $\hat{\eta}:\mathscr{E}(\widehat{\Adm(\mu)^K}) \to (\L,\preceq)$ be the edge labeling constructed in \S\ref{sec:2.3}. Let $J=\{\a \in\D_0 \mid \<\mu,\a \> = 0\}$.

We now provide an explicit description of the top two layers of $\Adm(\mu)^K$. 

\begin{proposition}\label{prop:top-two}Let $w \in \Adm(\mu)^K$ with $ \ell(w) = \<\mu,2\rho\> -1$. Let $z_1t^{\l}z_2^{-1}$ be an acute presentation of $w$ as in Proposition \ref{prop:acute-pres}. Then,
\begin{enumerate}
\item there is an edge $z_1\to z_2 $ in $\text{QBG}(\CR)$, $\l = \mu - \wt(z_1,z_2)$ and $\wt(z_1,z_2)\notin \Phi_J$. 

\item $w$ is covered by exactly two elements $t^{z_1(\mu)}$ and $t^{z_2(\mu)}$ in $\Adm(\mu)$.

\item Let $zW_J$ be the smallest element in $W/W_J$ such that $w \le t^{z(\mu)}$ and $t^{z(\mu)}\in\tW^K$. Then $\hat{\eta}(t^{z(\mu)} \gtrdot w)$ is a type II affine root. 
\end{enumerate}

\end{proposition}

\begin{remark}
Part (3) is essential for the proof of Theorem \ref{thm:main}. And this is the reason of the explicit construction of the edge labeling $\hat{\eta}$ in \S\ref{sec:2.3}.
\end{remark}

\begin{proof}
Let $t^{\mu'}\gtrdot w$ be a Bruhat covering, where $\mu' \in W_0(\mu)$, the $W_0$-orbit of $\mu$. Apply \cite[Proposition 4.5]{Sch24} and Lemma \ref{lem:trans}. There exists $\a\in\Phi^+$ such that one of the following conditions happens.

\begin{enumerate}[(i)]
\item $z_2 = z_1 s_\a$, $\l=\mu$, $\mu' = z_1(\mu)$ and there is a Bruhat edge $ z_1 \rightharpoonup z_2 $. In this case, $\hat{\eta}(t^{\mu'} \gtrdot w) = (z_1 (\a) , 0)$.

\item $z_2 = z_1 s_\a$, $\l=\mu - \a^{\vee}$, $\mu' = z_1(\mu)$, there is a quantum edge $ z_1 \rightharpoondown z_2 $. In this case, $\hat{\eta}(t^{\mu'} \gtrdot w) = ( z_1(\a) , 1)$.

\item $z_2 = z_1 s_\a$, $\l=\mu$, $\mu' = z_2 (\mu)$, there is a Bruhat edge $ z_1 \rightharpoonup z_2$. In this case, $\hat{\eta}(t^{\mu'} \gtrdot w) = (-z_1(\a), \< \mu,\a\>)$.
    
\item $z_2 = z_1 s_\a$, $\l=\mu - \a^{\vee}$, $\mu' = z_2 (\mu)$, there is a quantum edge $ z_1 \rightharpoondown z_2$. In this case, $\hat{\eta}(t^{\mu'} \gtrdot w) = (-z_1(\a), \< \mu,\a\>-1)$.
    
\end{enumerate}

We first prove that $\a\notin\Phi_J$.

In case (i) and (iii), as $z_1 t^{\mu} s_\a z_1^{-1} $ is an acute presentation of $w$, we get
$$  \Phi^- (z_1 (\a)) + \< \mu , \a\> - \Phi^-( z_1 s_\a (\a)) \ge 0.$$
Since $z_1(\a) \in \Phi^+$, we have $0 + \<\mu,\a\> -1  \ge 0$ and hence $\a\notin\Phi_J$. 

In case (ii) and (iv), as $z_1 t^{\mu-\a^{\vee}} s_\a z_1^{-1} $ is an acute presentation of $w$, we get
$$   \Phi^- (z_1(\a)) + \< \mu -\a^{\vee}, \a\> - \Phi^-( z_1 s_\a (\a)) \ge 0.$$
Since $z_1(\a) \in \Phi^-$, we have $1 + \<\mu,\a\> - 2 + 0 \ge 0$ and hence $\a\notin\Phi_J$. 

Note that statements (1) and (2) follow directly from the above explicit descriptions and the definition of $\hat{\eta}$. We now prove (3). 

We first consider the case where $ z_1 \to z_2 $ is a quantum edge. Then $z_2(\a)\in\Phi^+$. In this case, $\hat{\eta}(t^{z_2(\mu)} \gtrdot w) = (z_2 (\a) , \<\mu,\a\>-1)$ is a type II affine root and $t^{z_2(\mu)} \in \tW^K$ as desired.

We now consider the case where $ z_1 \to z_2 $ is a Bruhat edge. Then $z_1(\a) \in\Phi^+$. Suppose $t^{z_1(\mu)} \in \tW^K$, then $\hat{\eta}(t^{z_1(\mu)} \gtrdot w) = ( z_1(\a) , 1)$ is a type II root as desired. Now suppose $t^{z_1(\mu)} \notin \tW^K$, then we must have $t^{z_1(\mu)} = ws_{\tb} $ for some $\tb \in K$. Then $(z_1(\a),0) = \tilde\b$. In particular, $ z_1(\a) \in K\cap \Delta_0$. It follows that $\hat{\eta}(t^{z_2(\mu)} \gtrdot w) = ( -z_1(\a), \<\mu,\a\>)$ is a type II affine root. 
\end{proof}

\section{Proof of dual EL-shellability}\label{sec:4}

\subsection{Overview of the proof strategy}

We now prove Theorem \ref{thm:main}: The edge labeling $\hat{\eta}$ defined in \S\ref{sec:2.3} is a dual EL-labeling on $\widehat{\Adm(\mu)^K}$. This requires showing that for every interval in $\widehat{\Adm(\mu)^K}$:
\begin{itemize}
    \item There exists a unique label-increasing maximal chain.
    \item This chain is lexicographically minimal among all maximal chains.
\end{itemize}

For intervals $[w_1, w_2]^K$ with $w_1, w_2 \in \Adm(\mu)^K$, Proposition \ref{prop:Dyer} applies directly, giving us the desired properties. The critical case is the intervals $[w, \hat{1}]^K$ for $w \in \Adm(\mu)^K$. We shall first construct a label-increasing maximal chain from $\hat 1$ to $w$ that is lexicographically minimal and then show that it is indeed the unique label-increasing chain. It requires leveraging the quantum Bruhat graph and acute presentations developed in Section \S\ref{sec:3}. 

\subsection{The set $\Sigma_w^{J,K}$}\label{sec:4.1}
Let $w\in\Adm(\mu)^K$. In this subsection, we study the maximal elements of $\Adm(\mu)^K$ that are greater than or equal to $w$. 

Fix an acute presentation $ w=x t^{\l}y$ satisfying the condition in Proposition \ref{prop:acute-pres}. Then $t^{y^{-1}(\mu)}\in \tW^K$. Define
\begin{align*}
\Sigma_w  := \{ z \in W_0 \mid \wt(x, z ) + \wt( z,y^{-1}) \le \mu - \l\}.
\end{align*}
This set certainly depends on the choice of the acute presentation. 

By \cite[Corollary 3.4]{HY24}, the set $\{ z \in W_0 \mid  \wt(x,z) \le \mu - \l \}$ has a unique minimal element $z_{\min}$. We denote this element by $z_{\min}$. By Proposition~\ref{prop:admMembershipTest}, we get $\wt(x , y^{-1}) \le \mu -\l$. Then $y^{-1} \ge z_{\min}$ by definition of $z_{\min}$. Hence $\wt(z_{\min}, y^{-1}) = 0 $ by Lemma \ref{lem:Pos} (3). It follows that $z_{\min} \in \Sigma_w$. It is clear that $\Sigma_w \subseteq \{z \in W_0 \mid \wt(x,z) \le \mu - \l\}$. Thus $z_{\min}$ is the unique minimal element of $\Sigma_w$. 

Recall that $J  = \{ \a \in\D_0 \mid \<\mu,\a \>=0\}$ and that $W^{J,K} =\{a \in W^J \mid t^{a(\mu)} \in \tW^K\}$. Define
\begin{align*}
    \Sigma_w^{J,K}  &= \{a \in W^{J,K} \mid  w \le  t^{a(\mu)}\}.
\end{align*}
The set $\Sigma^{J, K}_w$ is in natural bijection with the set of maximal elements in $\Adm(\mu)^K$ that are greater than or equal to $w$. Next, we identify a distinguished element in $\Sigma_w^{J,K}$ that will serve as the gateway through which our label increasing chain passes.

\begin{lemma}\label{lem:min2}
The set $\Sigma_w^{J,K}$ has a unique minimal element $a_{\min}^K$. This element $a_{\min}^K$ is the unique element in $\pr(W_{K}) z_{\min} W_J\cap W^{J,K}$. Here, $\pr$ is the projection $\tW\to W_0$.
\end{lemma}
\begin{proof}
Set $\Sigma_w^J := \{a \in W^J \mid w \le t^{a(\mu)}\}$. Then $\Sigma_w^{J,K} = \{a \in \Sigma_w^J \mid t^{a(\mu)} \in \tW^K\}$. Let $a_1$ be the minimal representative in the coset $z_{\min}W_J$. For $a \in W^J$, by Theorem~\ref{thm:Felix} and Lemma~\ref{lem:trans}, $w\leq t^{a(\mu)}$ if and only if $au \in \Sigma_w$ for some $u\in W_J$. Therefore, $a_1$ is the unique minimal element in $\Sigma_w^J$. In particular, $a_1$ is a lower bound of $\Sigma_w^{J,K}$.

Suppose we are given an element $a_j \in \Sigma_w^J$ that is a lower bound of $\Sigma_w^{J,K}$. If $t^{a_j(\mu)} \in \tW^K$, then $a_j$ is the unique minimal element in $\Sigma_w^{J,K}$ as desired. Suppose $t^{a_j(\mu)}\notin \tW^K$. We construct $a_{j+1}$ with the same properties as follows.

Since $t^{a_j(\mu)}\notin \tW^K$, there is $\ta = (\a,k) \in K$ such that $t^{a_j(\mu)} s_{\ta} < t^{a_j(\mu)}$. In this case, we have $\< a_j(\mu),\a \> >0$, or equivalently, $a_j^{-1}(\a)\in \Phi^+ - \Phi_J$. Since $w\in \tW^K$ is assumed, we get
\begin{align*}
    w \leq t^{a_j(\mu)}s_{\ta}  < s_{\ta} t^{a_j(\mu)} s_{\ta}  =  t^{s_{\a} a_{j}(\mu)}.
\end{align*}
Define $a_{j+1}$ as the minimal representative in the coset $s_{\a} a_j$. Then $a_{j+1} \in \Sigma_w^J$. We claim that $a_{j+1}$ is also a lower bound of $\Sigma_w^{J,K}$.

Consider first the case where $\tilde \alpha  = (-\theta,1)$ for the highest root $\theta$ of an irreducible component of $\Phi$. Since $a_j^{-1}(-\th) \in \Phi^+$, we have $a_{j+1}\le s_{\th} a_j <a_j$. Hence $a_{j+1}$ is also a lower bound of $\Sigma_w^{J,K}$.

Consider next the case $\ta = (\a,0) \in \D_0$. Since $a_j^{-1}(\a)\in \Phi^+-\Phi_J$, we have $a_j < s_{\a} a_{j}$ and $s_{\a} a_{j} \in\tW^J$. Hence $a_{j+1} = s_{ \a} a_j$ by definition. Let $b \in \Sigma_w^{J,K}$. Since $t^{b(\mu)} \in \tW^K$, we have $\<b(\mu),\a\>\le0$ and hence $b^{-1}(\a) \in \Phi^- \cup \Phi_J^+$. In either case, it is straightforward to see that $a_{j+1}\le b$. This prove that $a_{j+1}$ is a lower bound of $\Sigma_w^{J,K}$.


We iterate this procedure until we end with an element $a_n\in \Sigma_w^{J,K}$ that is a lower bound of $\Sigma_{w}^{J,K}$ (the fact that $K$ is spherical easily implies that the procedure terminates). Then $a_n$ is the unique minimal element of $\Sigma_w^{J,K}$.

Finally, by the construction, we see that all $a_j$ lie in the double coset $\pr(W_K) z_{\min} W_J$. On the other hand, by \cite[Lemma 1.3]{HY24}, the intersection $\pr(W_{K}) z_{\min} W_J\cap W^{J,K}$ contains only one element.
\end{proof}

\subsection{Constructing the increasing chain}
We now construct a candidate for the unique increasing chain in $[w, \hat{1}]^K$. Let $a_{\min}^K$ be as in Lemma \ref{lem:min2}.
\begin{itemize}
    \item Interior segment: Consider the interval $[w, t^{a_{\min}^{K}(\mu)}]^K$. By Proposition \ref{prop:Dyer}, there exists a unique label-increasing maximal chain
\[
p_w: t^{a_{\min}^K(\mu)} \gtrdot w_r \gtrdot w_{r-1} \gtrdot \cdots \gtrdot w_1 = w
\]
in this interval, and this chain is lexicographically minimal among all chains from $t^{a_{\min}^{K}(\mu)}$ to $w$.

\item Augmented edge: Insert the edge $\hat{1} \gtrdot t^{a_{\min}^{K}(\mu)}$ labeled by $\eta_{a_{\min}^{K}}$ to obtain the full chain:
\begin{align}\label{eq:pw}
\hat p_w: \hat{1} \gtrdot t^{a_{\min}^{K}(\mu)} \gtrdot w_r \gtrdot w_{r-1} \gtrdot \cdots \gtrdot w_1 = w.
\end{align}
\end{itemize}

Since $a_{\min}^K$ is the minimal element in $W^{J,K}$ with $w \le t^{a_{\min}^K(\mu)}$, by Proposition \ref{prop:top-two} (3), the subchain $\hat{1} \gtrdot t^{a_{\min}^K(\mu)} \gtrdot w_{a_{\min}^K,r}$ is label-increasing. Hence, the concatenation \eqref{eq:pw} is label-increasing. By construction of $\hat{\eta}$, the chain \eqref{eq:pw} is lexicographically minimal among all maximal chains in $[w, \hat{1}]^K$.

\subsection{Uniqueness of the label increasing chain} To complete the proof of Theorem \ref{thm:main}, it remains to prove that \eqref{eq:pw} is the unique label-increasing maximal chain in $[w,\hat{1}]^K$. Let $a \in \Sigma_w^{J,K}$ with $a\ne a_{\min}^K$ and $$p_{a,w}: t^{a(\mu)} \gtrdot w_{a, r} \gtrdot w_{a, r-1} \gtrdot \cdots \gtrdot w_{a, 1} = w$$ be the unique label-increasing maximal chain in $[w,t^{a(\mu)}]^K$. It suffices to prove that $\hat{1}\gtrdot t^{a(\mu)} \gtrdot w_{a,r}$ is not label-increasing, or equivalently, $ t^{a(\mu)} \gtrdot w_{a,r}$ is labeled by a type I root. 

We shall construct an element $w'$ with 
\begin{enumerate}
    \item $t^{a(\mu)} \gtrdot w'$;
    \item $w' \ge w$ and
    \item $t^{a(\mu)}  \gtrdot w'$ is labeled by a type I root.
\end{enumerate}

Concatenating $t^{a(\mu)}\gtrdot w'$ with any chain from $w'$ to $w$, we obtain a chain from $t^{a(\mu)}$ to $w$ starting with a type I root. Thus, the lexicographically minimal chain $p_{a,w}$ also starts with a type I root. This then completes the proof of Theorem \ref{thm:main}.

The construction of $w'$ constitutes the heart of our argument, where all the previously developed results come together.

Note that $a \in \Sigma_w^{J,K}$. Let $u$ be a minimal element in $W_J$ such that $au \in \Sigma_w$. That is, $\wt(x , au) + \wt(au,y^{-1}) \le \mu - \l$. We consider two cases separately.
See Figure 1 and Figure 2 for the elements and paths in $\text{QBG}(\CR)$ being considered in case 1 and case 2 respectively. 

\begin{figure}[t]
\centering
\begin{minipage}{0.45\textwidth}
\centering

\begin{tikzpicture}
    \node (y) at (-2,2) {$y^{-1}$};
    \node (x) at (2,2) {$x$};
    \node (au) at (0,1) {$a u$};
    \node (z) at (-0.5,0) {$z$};
    
    \draw[->] (x) -- (au) ;
    \draw[->] (au) -- (z)node[right,midway] {\tiny $\a$};
    \draw[->] (z) -- (y);
\end{tikzpicture}
\caption{}
\end{minipage}
\hspace{0.1cm} 
\begin{minipage}{0.45\textwidth}
\centering
\begin{tikzpicture}
    \node (y) at (-2,1.9) {$y^{-1}$};
    \node (x) at (2,1.9) {$x$};
    \node (au) at (-.9,1) {$a u$};
    \node (z) at (-0.4,0) {$z$};
    \node (zmin) at (0.8  , -0.675) {$z_{min}$};
    \draw[->]  (au) -- (y);
    \draw[->] (x) -- (au) ;
    \draw[->] (z) -- (au)node[right,pos=0.4] {\tiny $\a$};
    \draw[->] (zmin) -- (z) ;
    \draw[->] (x) -- (zmin);
\end{tikzpicture}\vspace{-.2cm}
\caption{}
\end{minipage}\vspace{-.75cm}
\end{figure}

\textbf{Case 1:} $\wt( au,y^{-1}) >0$.

By Proposition \ref{prop:downup}, there is a path of \enquote{down-up type} from $au$ to $y^{-1}$. In particular, there is a quantum edge $au \rightharpoondown  z$ in $\QBG(\CR)$ such that $\wt(au, y^{-1}) =  \alpha^\vee + \wt(z,y^{-1})$, where $\alpha\in \Phi^+$ is the positive root satisfying $z =  aus_\alpha$. Since this is a quantum edge, we get $z < au$ and $ z(\alpha)\in \Phi^+$. By choice of $u$, this implies $\a \notin\Phi_J $ (as certainly $z \in \Sigma_w$). Therefore, $\<\mu,\a\> >0$. Set
\begin{align*}
    w' := t^{au (\mu)} s_{(-z (\alpha),1)} = au t^{\mu - \alpha^\vee } z^{-1}.
\end{align*}
Since $\<\mu,\a\> >0$, we have $w' < t^{au\mu}$. In particular, $\ell(w') \le \langle \mu,2\rho\rangle-1$. On the other hand, by the definition of quantum edges, we have $\ell(au)+\<\mu-\a^\vee, 2 \rho\>-\ell(z)=\langle \mu,2\rho\rangle-1$. By Corollary~\ref{cor:lengthViaAcutePres}, one deduces that the presentation $w'=au t^{\mu - \alpha^\vee } z^{-1}$ is acute and $\ell(w')=\langle \mu,2\rho\rangle-1$. This verifies the condition (1) for $w'$. 

We apply Theorem~\ref{thm:Felix} together with the observation
\begin{align*}
\wt(x, au) + \wt(z, y^{-1}) = \wt(x, au) + \wt(au, y^{-1}) - \alpha^\vee  \leq \mu-\l -  \alpha^\vee 
\end{align*}
to verify the condition (2) for $w'$. 

By \cite[Lemma 7.7 (3)]{Le16}, we have $\wt(s_\b z, y^{-1}) = \wt(z, y^{-1})$ for all $z\in W_0$ and $\b\in K\cap \Delta_0$. Hence $\wt(a' z, y^{-1}) = \wt(z, y^{-1})$ for all $a' \in W_{K \cap \D_0}$. As $\wt(au, y^{-1}) \neq \wt(z, y^{-1})$, $z(\alpha) \not \in \Phi_{K\cap \Delta_0}^+$. So $(-\a, 1)$ is a type I affine root. This verifies the condition (3) for $w'$. 




\textbf{Case 2:} $\wt( au,y^{-1}) =0$. This means $au \leq y^{-1}$. We get $z_{\min}\leq au$. This cannot be an equality, as we get $a\notin \mathrm{pr}(W_{K}) z_{\min} W_J$ from Lemma~\ref{lem:min2}.

Since $z_{\min}<au$, we find a maximal element $z'\in W_{K\cap \Delta_0} z_{\min}$ with $z' \leq au$ (this is the Deodhar lift from \cite{Deod87}). We saw above that $z_{\min}\notin W_{K\cap \Delta_0}au$, so $z' < au$. Now applying \cite[Lemma~7.4]{Lenart2015} to the set $K\cap \Delta_0\subseteq \Delta_0$ and the elements $w_0(au)^{-1}, w_0 (z')^{-1}\in W_0$, we find a lower Bruhat cover $z\lessdot au$ with $y\leq z$ and $z\notin W_{K\cap \Delta_0} au$. Let $\alpha\in \Phi^+$ with $z = aus_\alpha$. Thus, $z(\alpha)\in \Phi^+ - \Phi_{K\cap \Delta_0}^+$.

The conditions $z_{\min}\leq z \leq au\leq y^{-1}$ immediately imply $z\in \Sigma_w$. Hence $z\notin aW_J$ by choice of $u$, i.e.\ $\alpha\notin \Phi_J$. In particular, $\langle \mu,\alpha\rangle>0$. Set
\begin{align*}
    w' := t^{a\mu} s_{(-z\alpha, \langle \mu,\alpha\rangle\rangle)} = z t^\mu (au)^{-1}.
\end{align*}

Note that $\langle a(\mu),z(\alpha)\rangle = \langle au(\mu),aus_\alpha (\alpha)\rangle = -\langle \mu,\alpha\rangle<0$. Thus $w'<t^{a\mu}$. In particular, $\ell(w') \le \langle \mu,2\rho\rangle-1$. On the other hand, we have $\ell(z)+\<\mu, 2 \rho\>-\ell(au)=\langle \mu,2\rho\rangle-1$. By Corollary~\ref{cor:lengthViaAcutePres}, one deduces that the presentation $w' = zt^\mu (au)^{-1}$ is acute with $\ell(w') = \langle \mu,2\rho\rangle-1$. This verifies the condition (1) for $w'$.

We apply Theorem~\ref{thm:Felix} together with the observation
\begin{align*}
    \wt(x, z) + \wt(au, y^{-1}) = \wt(x, z) \leq \wt(x, z_{\min}) \leq \mu-\lambda 
\end{align*}
to verify the condition (2) for $w'$. 

We have shown that $z(\alpha)\notin \Phi_{K\cap \Delta_0}$. Thus, $(-z(\a),\langle \mu,\alpha\rangle)\in \Phi^+_{\aff}$ is a type I root. This verifies the condition (3) for $w'$. 

\section{Cohen-Macaulayness of Local Models}\label{sec:5}

\subsection{Notation}\label{sec:notation}
We adopt the conventions of \cite[§4.1--4.5]{PRS} and \cite[\S 8.1]{HH17}. Let $F$ be a nonarchimedean local field and $L$ the completion of the maximal unramified extension in a fixed separable closure $F^{\mathrm{sep}}/F$. Let $\G_0 \coloneqq \Gal(L^{\mathrm{sep}}/L)$.  

For a connected reductive $L$-group $G$, let $S \subseteq G$ be a maximal $L$-split torus with centralizer $T$---a maximal torus by Steinberg's theorem. Denote the absolute and relative Weyl groups by $W \coloneqq N_G(T)(L^{\mathrm{sep}})/T(L^{\mathrm{sep}})$ and $W_0 \coloneqq N_G(T)(L)/T(L)$, respectively. 

The reduced root datum $\Sigma \coloneqq (X^*, X_*, \Phi, \Phi^\vee)$ of $(G,T)$ yields an affine Weyl group $W_{\mathrm{aff}}(\Sigma)$ isomorphic to the Iwahori-Weyl group $\widetilde{W}(G_{\mathrm{sc}})$ of the simply-connected cover $G_{\mathrm{sc}}$. Let $V \coloneqq X_*(T)_{\G_0} \otimes \mathbb{R} \cong X_*(S) \otimes \mathbb{R}$ denote the apartment for $S$, equipped with a fixed special vertex and the alcove $\bar{\mathbf{a}}$ in the antidominant chamber whose closure contains this vertex.

The Iwahori-Weyl group $\widetilde{W}(G)$ decomposes as: $X_*(T)_{\G_0} \rtimes W_0$ and $W_{\mathrm{aff}}(\Sigma) \rtimes \Omega$, where $\Omega \subseteq \widetilde{W}(G)$ stabilizes $\bar{\mathbf{a}}$. The torsion subgroup $X_*(T)_{\G_0,\mathrm{tors}} \subseteq X_*(T)_{\G_0}$ is central in $\widetilde{W}(G)$, as $W_0$ acts trivially on it.

Let $(-)^\flat$ denote the quotient by $X_*(T)_{\G_0,\mathrm{tors}}$. This induces isomorphisms:
\[
\widetilde{W}(G)^\flat \cong X_*(T)^\flat_{\G_0} \rtimes W_0 \cong W_{\mathrm{aff}}(\Sigma) \rtimes \Omega^\flat,
\]
where $\Sigma$ is reinterpreted as a root datum with $X_* = X_*(T)^\flat_{\G_0}$. For $x \in \widetilde{W}(G)$, write $x^\flat$ for its image in $\widetilde{W}(G)^\flat$.

For a $G$-conjugacy class $\{\mu\} \subseteq X_*(G) \cong X_*(T)/W$, define $\widetilde{\Lambda}_{\{\mu\}}$ to be $B$-dominant representatives in $\{\mu\}$ for some $L$-rational Borel $B \supset T$ and $\Lambda_{\{\mu\}}$ to be the image of $\widetilde{\Lambda}_{\{\mu\}}$ in $X_*(T)_{\G_0}$. Define $$\Adm(\{\mu\}) \coloneqq \{ x \in \widetilde{W}(G) \mid x \leq t^\lambda \ \text{for some} \ \lambda \in \Lambda_{\{\mu\}} \}$$
with Bruhat order $\leq$ from the decomposition $\widetilde{W}(G) = W_{\mathrm{aff}}(\Sigma) \rtimes \Omega$. 

The following result is proved in \cite[\S 8.2]{HH17}.
\begin{lemma}\label{lem:flat_adm}
\begin{enumerate}
    \item For $\tau \in \Omega$ and $x,y \in W_{\mathrm{aff}}(\Sigma)\tau$: $x \leq y \iff x^\flat \leq y^\flat$.
    \item If $\Lambda_{\{\mu\}} \subseteq W_{\mathrm{aff}}(\Sigma)\tau$, then $\Adm(\{\mu\})$ is the preimage of $\Adm(\Lambda^\flat_{\{\mu\}})$ under $W_{\mathrm{aff}}(\Sigma)\tau \to \widetilde{W}(G)^\flat$, where $\Lambda^\flat_{\{\mu\}} \subseteq X_*(T)^\flat_{\G_0}$.
\end{enumerate}
\end{lemma}

\subsection{Schubert varieties} Let $K$ be our subset of $\D_{\aff}$ as before, and denote for each $w \in \tW^K$ the associacted Schubert variety in the partial affine flag variety by $S_w$. In the mixed characteristic case or in the equal characteristic case with $p \nmid \vert\pi_1(G_{\text{der}}) \vert$, all these Schubert varieties are normal and Cohen-Macaulay. It is discovered in \cite[Theorem 1.1]{HLR} that in the equal characteristic case with $p \mid \vert \pi_1(G_{\text{der}}) \vert$, only finitely many Schubert varieties are normal, and any non-normal Schubert variety is not Cohen-Macaulay. In this case, one needs to pass to the seminormalization. Following the stacks project, the seminormalization $\tilde S_w$ is the initial scheme mapping universally homeomorphically to $S_w$ with the same residue field. The following result is due to Fakhruddin-Haines-Louren\c co-Richarz \cite[Theorem 4.1]{FHLR}. 

\begin{theorem}\label{thm:CMforS}
    The seminormalized Schubert variety $\tilde S_w$ is normal and Cohen-Macaulay. 
\end{theorem}

\subsection{Local models}
In this subsection, we consider the local models with parahoric level structure. Let $\bG$ be a connected reductive group over $F$ and $\{\mu\}$ be a (not necessarily minuscule) conjugacy class of geometric cocharacters defined over the reflex field $E$, a finite separable extension of $F$. 


A uniform definition of these local models is not available; their construction has developed through several key works. The foundational group-theoretic construction for tamely ramified groups was established by Pappas–Zhu \cite{PZ13}. This was later extended beyond the tamely ramified case by Levin \cite{Le16}, Lourenço \cite{Lo19+}, and Fakhruddin–Haines–Lourenço–Richarz \cite{FHLR}. In equal characteristic, a construction for arbitrary groups was given by Richarz \cite{Ri16}, while the mixed-characteristic case has been largely addressed by the aforementioned authors. We note that in mixed characteristic, constructions depend on the choice of a parahoric group scheme lifting. Finally for minuscule $\mu$, the work of Ansch\"utz-Gleason-Lourenço-Richarz \cite{AGLR} provides a unique projective flat weakly normal scheme representing Scholze-Weinstein's diamond local model \cite[Conjecture 21.4.1.]{SW20}.

For our purposes, we consider local models satisfying the following essential property (cf.\ \cite[Conjecture~2.13]{HPR20}):\\

\begin{minipage}{.05\textwidth}$(\ast)$\end{minipage}\begin{minipage}{.9\textwidth}
The local $\tilde M_{\{\mu\}}$ attached to $(\bG, K, \{\mu\})$ is a flat scheme, whose generic fibre is the Schubert variety $\tilde S$ in the $B_{dR}^+$-affine Grassmannian of $\bG$ corresponding to $\{\mu\}$, and the reduced special fibre is equal to $\cup_{w \in \Adm(\{\mu'\})^K} \tilde S_w'$.
\end{minipage}\\

Here $(\bG', K, \{\mu'\})$ is the equicharacteristic analogues of $(\bG, K, \{\mu\})$ (see \cite[\S 2]{FHLR}), $\breve I'$ is the standard Iwahori subgroup, $\breve P' = \breve I'W_K\breve I'$ the parahoric of $K$ and $\cup_{w \in \Adm(\{\mu'\})^K} \tilde S_w'$ is the $\mu'$-admissible locus in the equicharacteristic partial affine flag variety. The local models constructed in \cite{PZ13, Le16, Lo19+, FHLR, Ri16} as well as most local models constructed in \cite{AGLR} satisfy property $(\ast)$.

\begin{theorem}\label{thm:CM}
    Let $\tilde M_{\{\mu\}}$ be a local model for $(\bG, K, \{\mu\})$ satisfying the property $(\ast)$. Then $\tilde M_{\{\mu\}}$ is Cohen-Macaulay. 
\end{theorem}

\begin{proof}
We consider the seminormalized Schubert variety $\tilde S'_w$ in the partial affine flag variety of $\bG'$. By Theorem \ref{thm:CMforS}, for any $w$, $\tilde S'_w$ is Cohen-Macaulay. Let $X=\cup_{w \in \Adm(\{\mu'\})} \tilde S'_{w}$. By Theorem~\ref{thm:main} and Proposition~\ref{prop:dualELImpliesCM}, $\Adm(\{\mu'\})^K$ is $N$-Cohen-Macaulay. Hence, by \cite[Proposition 4.24]{Go01}, $X$ is Cohen-Macaulay.  

Note that the generic fibre of $\tilde M_{\{\mu\}}$ is the seminormalization of a single Schubert variety, and hence is Cohen-Macaulay by Theorem \ref{thm:CMforS}. By \cite[Lemma 5.7]{HR23}, the whole local model $\tilde M_{\{\mu\}}$ is Cohen-Macaulay. 
\end{proof}
A similar proof yields the following result.
\begin{proposition}\label{prop:SWisCM}
    The local models characterized by Scholze-Weinstein \cite{SW20} and constructed by Anschütz-Gleason-Lourenço-Richarz \cite{AGLR} are Cohen-Macaulay.
\end{proposition}
\begin{proof}
    Flatness is shown in \cite[Theorem~1.2]{AGLR}, which also gives a description of the generic fibre and implies reducedness of the special fibre. The special fibre is described in all cases by Cass-Lourenço \cite[Corollary~1.5]{CL25} as the canonical deperfection of the admissible locus inside the partial flag variety of $G$. They show in \cite[Theorem~4.15]{CL25} that the scheme $A_{\mathcal G,\mu,1}$, which is identified with the special fibre of the local model in \cite[Corollary~4.16]{CL25}, is the union of deperfected Schubert varieties $\mathcal{F\ell}_{\mathcal G,\leq t^{a(\mu)},1}$. These deperfected Schubert varieties are integral Cohen-Macaulay schemes with closure relations given by the Bruhat order. Now we can argue as in the proof of Theorem~\ref{thm:CM}.
\end{proof}
We are able to reprove the following known result.
\begin{corollary}\label{cor:KPZisCM}
    The integral models of Shimura varieties constructed by Kisin-Pappas-Zhou \cite{KPZ24} are Cohen-Macaulay.
\end{corollary}
\begin{proof}
    Indeed, by \cite[Theorem~1.1.2]{KPZ24}, their integral models are \'etale-locally isomorphic to the local models considered in Proposition~\ref{prop:SWisCM}.
\end{proof}

Combining Theorem \ref{thm:main} with \cite[Lemma 4.22]{Go01}, we have the following result. 

\begin{theorem}\label{thm:CMConstructive}
    Let $C\subseteq W^{J,K}$ be a subset such that for all $a\in C$ and all $b\in W^{J,K}$ with $a>b$, we have $b\in C$. In other words, $C$ is the intersection of $W^{J,K}$ with any union of lower Bruhat intervals $[1,a], a\in W_0$. Then the subvariety $\cup_{v \in C} \tilde S'_{t^{v(\mu)}}$ of the partial affine flag variety is Cohen-Macaulay.
\end{theorem}

\subsection{Further questions} 

Let $\s$ be a length-preserving automorphism of $\tW$. Then every element $w\in \tW$ can be written in the form
\begin{align*}
    w = \tau s_1\cdots s_{\ell(w)}
\end{align*}
for $\tau\in \tW$ of length zero and simple reflections $s_1,\dotsc,s_{\ell(w)}\in \tilde{\BS}$. Note that conjugation by $\tau$ induces an automorphism $\Ad(\tau) : \tilde{\BS}\to \tilde{\BS}$. We say that $w$ is a \emph{partial $\sigma$-Coxeter element} if the $s_i$ lie in pairwise distinct $\sigma\circ \Ad(\tau)$-orbits. We call it a \emph{spherical $\sigma$-Coxeter element} if moreover there is at least one $\sigma\circ\Ad(\tau)$-orbit of $\tilde{\BS}$ not containing any $s_i$. These two notions are independent of the choice of reduced word for $w$. Set
\begin{align*}
    \Cox(\mu)^K := \{w\in \Adm(\mu)^K\mid w\text{ is a spherical $\sigma$-Coxeter element}\}.
\end{align*}
For the \emph{basic loci of Coxeter type} in the sense of Görtz-He-Nie \cite{GHN3}, this is the indexing set of the EKOR stratification of said basic locus. This is a group theoretic model of the basic locus of the special fibre of certain very nicely behaved Shimura varieties. We formulate the following conjecture based on empirical observations.

\begin{conjecture} For any tuple $(\tW,\s,\mu,K)$, the poset $\widehat{\Cox(\mu)^K} = \Cox(\mu)^K \sqcup \{\hat{1}\}$ is dual EL-shellable.\footnote{The partial order in $\Cox(\mu)^K$ is the refined order $\le_{K,\s}$. However, by \cite[Proposition~2.8]{GHN3}, in the Coxeter type case, this refined order is equivalent to the usual Bruhat order.}
\end{conjecture}
In the Coxeter type case $(A_4, \s = \id, \mu = \omega_1^{\vee}+\omega_{4}^{\vee} , K=\D_0)$, the set $\Cox(\mu)^K$ is depicted below. It is easy to construct the desired labeling on this set. 

This conjecture would imply that Cohen-Macaulayness of a basic locus of Coxeter type may simple be checked for each EKOR stratum individually.\\

\centering
\begin{tikzpicture}[yscale=0.8,xscale=1.2]
    \node (empty) at (0,0) {$1$};
    \node (0) at (0,1) {$s_0$};
    \node (04) at (-1,2) {$s_{0}s_{4}$};
    \node (01) at (1,2) {$s_{0}s_{1}$};
    \node (043) at (-2 , 3) {$s_0s_4s_3$};
    \node (041) at (0, 3) {$s_0s_4s_1$};
    \node (012) at (2,3) {$s_0s_1s_2$};
    \node (0432) at (-3,4) {$s_0s_4s_3s_2$};
    \node (0431) at (-1,4) {$s_0s_4s_3s_1$};
    \node (0412) at (1,4) {$s_0s_4s_1s_2$};
    \node (0123) at (3,4) {$s_0s_1s_2s_3$};

    \draw[-]  (empty) -- (0);
    \draw[-]  (0) -- (04);
    \draw[-]  (0) -- (01);
    \draw[-]  (04) -- (043);
    \draw[-]  (04) -- (041);
    \draw[-]  (01) -- (041);
    \draw[-]  (01) -- (012);
    \draw[-]  (043) -- (0432);
    \draw[-]  (043) -- (0431);
    \draw[-]  (041) -- (0431);
    \draw[-]  (041) -- (0412);
    \draw[-]  (012) -- (0412);
    \draw[-]  (012) -- (0123);

\end{tikzpicture}

\printbibliography
\end{document}